\documentclass[12pt, a4paper]{amsart}
\usepackage[left=3cm, right=3cm, top=2cm, bottom=2cm]{geometry}
\usepackage{amsmath,amssymb,amsthm,enumerate}
\usepackage{graphicx}
\usepackage{pstricks-add}
\usepackage{mathrsfs} 
\usepackage{xspace}

\usepackage[pagebackref=true]{hyperref}

\title[A radius~$1$ irreducibility criterion]{A radius~$1$ irreducibility criterion\\ for lattices in products of trees}

\author{Pierre-Emmanuel Caprace}
\thanks{P.-E.C. is a F.R.S.-FNRS senior research associate.}
\address{Universit\'e catholique de Louvain, IRMP, Chemin du Cyclotron 2, bte L7.01.02, 1348 Louvain-la-Neuve, Belgique}
\email{pe.caprace@uclouvain.be}

\date{August 10, 2021}

\newtheorem{thm}{Theorem}[section]

\newtheorem{prop}[thm]{Proposition}
\newtheorem{lem}[thm]{Lemma}
\newtheorem{cor}[thm]{Corollary}

\newtheorem*{claim*}{Claim}

\theoremstyle{definition}

\newtheorem{rmk}[thm]{Remark}

\newcommand{\Sym}{\mathrm{Sym}}
\newcommand{\Alt}{\mathrm{Alt}}
\newcommand{\Aut}{\mathrm{Aut}}
\newcommand{\Out}{\mathrm{Out}}

\newcommand{\soc}{\mathrm{soc}}

\newcommand{\triv}{\{1\}}

\newcommand{\PSL}{\mathrm{PSL}}
\newcommand{\PGL}{\mathrm{PGL}}
\newcommand{\PGammaL}{\mathrm{P}\Gamma\mathrm{L}}

\newcommand{\FF}{\mathbf{F}}

\begin{document}

\begin{abstract} 
Let  $T_1, T_2$ be regular trees of degrees $d_1, d_2 \geq 3$. Let also   $\Gamma \leq \Aut(T_1) \times \Aut(T_2)$ be a group acting freely and transitively on $VT_1 \times VT_2$. For $i=1$ and $2$, assume that the local action of $\Gamma$ on $T_i$ is $2$-transitive;  if moreover $d_i \geq 7$, assume that the local action contains $\Alt(d_i)$. We show that  $\Gamma$ is   irreducible, unless $(d_1, d_2)$ belongs to an explicit small set of exceptional values.  This yields an irreducibility criterion for 
$\Gamma$ that can be checked purely in terms of its local action on a ball of radius~$1$ in $T_1$ and $T_2$.  Under the same hypotheses, we  show moreover that if $\Gamma$ is irreducible, then it is hereditarily just-infinite, provided the local action on $T_i$ is not the affine group $\mathbf F_5 \rtimes \mathbf F_5^*$. The proof of irreducibility relies, in several ways, on the Classification of the Finite Simple Groups.
\end{abstract}
\maketitle


\section{Introduction}

%
%
%

The study of lattices in products of trees was pioneered by D.~Wise \cite{Wise_PhD} and M.~Burger and S.~Mozes \cite{BM-cras}, \cite{BuMo2}. Their seminal works revealed that the class of finitely generated groups admitting a Cayley graph that is isomorphic to the Cartesian product of two trees is very rich: it contains not only products of virtually free groups, but also certain  $S$-arithmetic groups and some finitely presented virtually simple groups, among many others. Such groups are called \textbf{BMW-groups} and form a special class of lattices in products of trees. An introduction to this fascinating subject may be consulted in \cite[Section~4]{Cap_survey}.  
The goal of this paper is to present a sufficient condition, that is straightforward to check in practise, ensuring that a BMW-group is irreducible. 

Let $T_1, T_2$ be locally finite trees and  $\Gamma \leq \Aut(T_1) \times \Aut(T_2)$ be a group acting with finite stabilizers and finitely many orbits. Equivalently $\Gamma$ is a discrete subgroup of $\Aut(T_1) \times \Aut(T_2)$ acting cocompactly on $T_1 \times T_2$. Such a group $\Gamma$ is called a \textbf{cocompact lattice} in the product $T_1 \times T_2$. Since we only consider cocompact lattices in this paper,   the adjective \emph{cocompact} will henceforth be omitted. We say that $\Gamma$ is \textbf{reducible} if it contains a finite index subgroup of the form $K_1 \times K_2$, where $K_i \leq \Gamma$ acts trivially on $T_{3-i}$ and freely and cocompactly on $T_i$ for $i=1$ and $2$. Otherwise $\Gamma$ is called \textbf{irreducible}.  
Determining whether a given lattice is reducible is a crucial basic question, and there is no known algorithm   deciding if that property holds in full generality. Burger and Mozes observed however that the irreducibility of $\Gamma$ can   be tested in an efficient algorithmic way under an extra hypothesis on the \textit{local action} of $\Gamma$ on $T_1$ or $T_2$. We recall that, given a group $G$ acting on a graph $X$ by automorphisms,  the \textbf{local action of level~$n$} of $G$  at a vertex $v$ is the action of the stabilizer $G_v$  on the $n$-ball around $v$. The local action of level~$1$ is simply called the  \textbf{local action} for short. We say that the local action of $G$ has a property $\mathcal P$ (e.g. is transitive, primitive, $2$-transitive, etc.) if the local action of $G$ at every vertex has property $\mathcal P$. If $G$ is vertex-transitive, then the local actions of $G$ at all vertices are pairwise isomorphic; in that case, the corresponding abstract permutation group is called the \textbf{local action of $G$ on $X$}. An important fact due to Burger and Mozes \cite[\S3.3]{BuMo1}, \cite[\S5]{BuMo2} is that, if  $d_i \geq 6$, if $G$ is vertex-transitive on $T_i$ and if the local action of $\Gamma$ on $T_i$ contains $\Alt(d_i)$ for  $i= 1$ or $2$, then the irreducibility of $\Gamma$ can be tested by considering the local action of level~$2$ of $\Gamma$ on $T_i$. More generally, using the work of V.~Trofimov and R.~Weiss \cite{TrofimovWeiss}, one can show that, for all $d_i \geq 3$, if the  local action of $\Gamma$ on $T_i$ is $2$-transitive, then the irreducibility of $\Gamma$ can be tested by considering the local action of level $7$ of $\Gamma$ on $T_i$ (see \cite[Corollary~4.12]{Cap_survey}). 

 An action of a group on a set is called \textbf{regular} if it is free and transitive. In the present paper, we focus on the special class of lattices in $T_1 \times T_2$ formed by the groups $\Gamma \leq  \Aut(T_1) \times \Aut(T_2)$ acting regularly on $VT_1 \times VT_2$. In that case the tree $T_i$ is regular of degree $d_i$. Following \cite[Section~4]{Cap_survey}, such a group $\Gamma$ is called a \textbf{BMW-group of degree $(d_1, d_2)$}. When this is the case, the group $\Gamma$ has a generating set $S$ such that the Cayley graph of $(\Gamma, S)$ is the Cartesian product $T_1 \times T_2$. 
This paper was initiated by the following observation, which follows easily from the aforementioned work of Trofimov--Weiss. It shows that, in principle, when the local action on \emph{both} tree factors is $2$-transitive, then the lattice $\Gamma$ is ``almost always'' irreducible. 

\begin{thm}\label{thm:NonExplicit}
Let $d_1 \geq  d_2 \geq 3$, let 	  $T_1, T_2$ be regular trees of degrees $d_1, d_2$ and let $\Gamma \leq \Aut(T_1) \times \Aut(T_2)$ be a  group acting  regularly on the vertices of $T_1 \times T_2$.  Assume  that for $i=1, 2$, the local action $F_i$ of $\Gamma$ on $T_i$ is $2$-transitive. If 
$$d_1 \geq (d_2!)\big((d_2-1)!\big)^{\frac{d_2 ((d_2-1)^5-1 )}{d_2-2}},$$
then $\Gamma$ is irreducible.
\end{thm}
 
A similar idea is used by C.~H.~Li in his proof of \cite[Theorem~1.1]{Li_procAMS}.
 
The condition that the $\Gamma$-action on $VT_1 \times VT_2$ be free is essential in Theorem~\ref{thm:NonExplicit}. Indeed, given any $d \geq 3$, let $W_d$ be the free product of $d$ copies of the cyclic group of order~$2$. Then $\Sym(d)$ acts by automorphisms on $W_d$ by permuting the $d$ generators of order~$2$. Therefore, for all $d_1 \geq d_2 \geq 3$, the direct product 
$$\Gamma = (W_{d_1} \rtimes \Sym(d_1))\times (W_{d_2} \rtimes \Sym(d_2))$$
is an obviously reducible lattice in $T_1 \times T_2$, where $T_i$ is the regular tree of degree $d_i$. Its local action on $T_i$ is $\Sym(d_i)$. Moreover $\Gamma$ acts transitively, but not freely, on $VT_1 \times VT_2$, showing that the hypothesis of freeness of the $\Gamma$-action cannot be removed in  Theorem~\ref{thm:NonExplicit}. 
 
Under extra assumptions on the local action, the bound contained in Theorem~\ref{thm:NonExplicit} can be vastly improved. This is illustrated by   the following result, where $\mathbf{C}_n$ denotes the cyclic group of order $n$. 

\begin{thm}\label{thm:Main}
Let  $d_1 \geq  d_2 \geq 3$, let 	  $T_1, T_2$ be regular trees of degrees $d_1, d_2$ and let $\Gamma \leq \Aut(T_1) \times \Aut(T_2)$ be a group acting regularly on the vertices of $T_1 \times T_2$.  Assume  that for $i=1, 2$, the local action $F_i$ of $\Gamma$ on $T_i$ is $2$-transitive. Assume moreover that if $d_i \geq 7$, then $F_i \geq  \Alt(d_i)$.  Then $\Gamma$ is irreducible provided   none of the following conditions holds:
\begin{enumerate}[(i)]
	\item $d_2 = 3$, and 
	$$d_1 \in \big\{23, 24,47\big\}.$$

	\item $d_2 = 4$, and 
	$$ d_1 \in \big\{6n \mid n \geq 2  \text{ divides } 972\big\} \cup \big\{ 12n -1\mid n  \text{ divides } 972\big\}.$$  	

	\item $d_2 =  5$, $F_2 \cong \mathbf{C}_5 \rtimes \mathbf{C}_4$,  and  
$$d_1 \in \big\{10,19,20,39, 40,79\big\}.$$

	\item $d_2 =  5$, $\soc(F_2) \cong \Alt(5)$,  and  
	$$d_1 \in \big\{30n \mid n  \geq 2 \text{ divides } 768\big\} \cup \big\{60n-1 \mid n  \text{ divides } 768\big\}.$$ 	

		\item $d_2 = 6$, $\soc(F_2) \cong \Alt(5)$, and 
	$$d_1 \in \big\{30n  \mid n \geq 2 \text{ divides } 200\big\} \cup \big\{60n-1 \mid n \text{ divides } 200\big\}.$$
	\item $d_2 \geq 6$, and 
\begin{align*}
d_1  \in  & \ 
  \big\{\frac{d_2!} 2-1 ,\frac{d_2!} 2,  d_2! -1,   \frac{d_2!(d_2-1)!} 4-1,  \frac{d_2!(d_2-1)!} 4, \\ 
  & \quad  
  \frac{d_2!(d_2-1)!} 2-1,  \frac{d_2!(d_2-1)!} 2, d_2!(d_2-1)!-1\big\}. 
\end{align*}
\end{enumerate}
\end{thm} 

The \textbf{socle} of a finite group $F$, denoted by $\soc(F)$, is the subgroup generated by all the minimal normal subgroups of $F$. While contemplating the list of possible exceptions in small degree in Theorem~\ref{thm:Main}, it is useful to keep in mind the list of finite $2$-transitive groups of degree~$\leq 6$, which is   recalled  in Table~\ref{tab:Small2Trans} below.

Theorem~\ref{thm:Main} provides in particular an irreducibility criterion for a BMW-group $\Gamma$ of degree $(d_1, d_2)$: if the pair $(d_1, d_2)$ is not one of the exceptions from the list (i)--(vi) in the theorem, then $\Gamma$ is irreducible provided its local action on $T_i$ is $2$-transitive and, in case $d_i \geq 7$, if it also contains $\Alt(d_i)$, for $i=1$ and $2$. That criterion depends only on the local actions of level~$1$, and is thus considerably easier to use in practise than the other criteria mentioned above.

Notice the contrast between the bound on $d_1$ in Theorem~\ref{thm:NonExplicit} and the range of values for $(d_1, d_2)$ in Theorem~\ref{thm:Main}. 
Examples of  reducible lattices $\Gamma$ as in Theorem~\ref{thm:Main} with $(d_1, d_2) = (23, 3)$, $(24,3)$,  $(47, 3)$,   $(11663, 4)$, $(19,5)$, $(39, 5)$ and $(79, 5)$ will be highlighted, relying on the work of Xu--Fang--Wang--Xu \cite{XFWX}, Li--Lu \cite{LiLu},  M.~Conder~\cite{Conder}   and Ling--Lou~\cite{LingLu1, LingLu2}, see Proposition~\ref{prop:Existence} below. In particular Theorem~\ref{thm:Main} is sharp in the case $d_2 = 3$. It would be very interesting to determine which of the exceptional values occurring in Theorem~\ref{thm:Main} are indeed realized by actual examples of reducible lattices (for $4 \leq d_2 \leq 6$, not all values of $d_1$ appearing in the theorem can be realized, see Remark~\ref{rem:sharpen}), or at least whether infinitely many values of $(d_1, d_2)$ with $d_2 \geq 6$ can occur. As we shall see in Section~\ref{sec:FiniteReduction}, this is a question in finite group theory. The examples found for the small values of $d_2$ provide evidence for a positive answer to the latter question.

We point out the following immediate consequence.

\begin{cor}\label{cor:SmallDegree}
	Let    $T_1, T_2$ be regular trees of degrees $ d_1, d_2 \in \{3,4,5,6\}$ and let $\Gamma \leq \Aut(T_1) \times \Aut(T_2)$ be a group acting regularly on the vertices of $T_1 \times T_2$.  Assume  that for $i=1, 2$, the local action   of $\Gamma$ on $T_i$ is $2$-transitive. Then $\Gamma$ is irreducible.
\end{cor}

The hypothesis of the regularity  of the  $\Gamma$-action  on $VT_1 \times VT_2$ is essential: indeed, there are examples   of reducible lattices $\Gamma \leq \Aut(T_1) \times \Aut(T_2)$ acting  with $2$ orbits of vertices and locally $2$-transitive actions at every vertex of both factors,  with $(d_1, d_2) =(3,4)$, see Remark~\ref{rem:Small} below.  

Using the basic covering theory of graphs (see Proposition~\ref{prop:Infinite-Finite}), one shows that Theorem~\ref{thm:Main} is equivalent to a statement on finite groups acting on graphs, namely Theorem~\ref{thm:MainFinite} below. The proof of the latter statement relies heavily, and in several ways, on the Classification of the Finite Simple groups. Particularly relevant is the classification, due to Liebeck--Praeger--Saxl \cite[Corollary~5]{LPS}, of all pairs $(G, M)$ consisting of a finite almost simple group $G$ and a subgroup $M \leq G$ whose order involves all primes dividing the order of $G$ (see   Section~\ref{sec:LPS} below).  It is moreover closely related to the well studied field of {finite groups admitting an $s$-arc transitive Cayley graph} (see \cite{LiXia} and references therein).


Combining the work of Burger--Mozes \cite{BuMo1}, Bader--Shalom \cite{BaSha} and V.~Trofimov \cite{Tro}, we will show that if a  lattice  $\Gamma$ satisfies  the hypotheses of Theorem~\ref{thm:Main} and if it is irreducible, then it is \textbf{hereditiarily just-infinite}, i.e. 
$\Gamma$ is infinite, and every proper quotient of every finite index subgroup of $\Gamma$ is finite. 

\begin{cor}\label{cor:HJI}
Retain the hypotheses of Theorem~\ref{thm:Main} and assume that $\Gamma$ is irreducible. Assume moreover that  if $d_i = 5$ then $F_i \not \cong \mathbf{C}_5 \rtimes \mathbf{C}_4$ for $i=1$ and $2$. Then $\Gamma$ is hereditarily just-infinite.
\end{cor}

See Corollary~\ref{cor:ji-crit} for a more general statement. 

Numerous explicit examples of BMW-groups of small degree satisfying the hypotheses of Theorem~\ref{thm:Main} and Corollary~\ref{cor:HJI} are given in \cite[\S4]{Cap_survey}, \cite{Radu_Latt} and \cite{Rattaggi_PhD}. 

\subsection*{Acknowledgements}

I thank Nicolas Radu for numerous discussions and for performing inspiring experiments with an exhaustive list of BMW-groups of small degree (some information about that work can be found in \cite{Radu_Latt} and  \cite{Radu_PhD}).  
I am  grateful to Michael Giudici for a clarification about a factorization of $\Omega_8^+(2)$ that appears in Case (7) of the proof of Lemma~\ref{lem:N-simple}, and for pointing out the   reducible examples  in degree $(19, 5), (39, 5)$ and $(79, 5)$ that are recorded in Proposition~\ref{prop:Existence}(v). I also thank Marston Conder for his remarks related to that proposition. I am grateful to both referees for their comments and corrections.

\section{Preliminaries}

\subsection{Groups acting on graphs and local action}\label{sec:Graphs}

For graphs and trees, we use the terminology and notation of \cite[\S2.1]{BL}. A graph $X$ consists of a set of vertices $VX$, a set of oriented edges $EX$,  two maps $\partial_0, \partial_1 \colon EX \to V$ representing the endpoints of edges, and an orientation reversing map $EX \to EX : e \mapsto \bar e$ satisfying $\partial_i \bar e = \partial_{1-i} e$ and $\bar{\bar e}= e \neq \bar e$. For $x \in VX$ we set $E(x) = \{e \in EX \mid \partial_0(e)= x\}$. A \textbf{geometric edge} of $X$ is a pair $\{e, \bar e\}$ with $e \in EX$. 

Let now $G$ be group acting on a graph $X$ by automorphisms. We denote by $G_x$ the stabilizer of an element $x \in VX \cup EX$. For $x \in VX$ and $m \geq 0$, we also denote by $G_x^{[m]}$ the subgroup of $G$ fixing all vertices $y \in VX$ at distance $d(x, y) \leq m$. The quotient group $G_x/G_x^{[1]}$, viewed as a permutation group on $E(x)$, is the \textbf{local action} of $G$ at $x$. More generally, the group $G_x/G_x^{[m]}$, viewed as a permutation group on the $m$-ball around $x$, is called the  \textbf{local action of level $m$} of $G$ at $x$. 

An \textbf{edge inversion} is an element   $g \in G$ such that $ge = \bar e$ for some $e \in EX$. If $G$ acts without edge inversion, then  we can form the quotient graph $G \backslash X$, see \cite[\S2.2]{BL}. We say that the $G$-action on $X$ is \textbf{free} if $G$ acts freely on $VX$ and freely on the set of geometric edges. Equivalently, the $G$-action on $X$ is free if $G$ acts freely on $VX$ and has no edge inversion.

\begin{lem}\label{lem:Kernel}
	Let $X$ be a connected graph and $G \leq \Aut(X)$ be a group of automorphisms. Given a normal subgroup $N$ of $G$ acting freely on $X$, the kernel of the $G$-action on the quotient graph $N\backslash X$ coincides with $N$.
\end{lem}
\begin{proof}
Let $g \in G$ act trivially on the quotient graph $N\backslash X$ and let $x \in VX$. Since $gN(x) = N(x)$, there exists $n \in N$ with $gn(x) = x$. Let now $y$ be any vertex of $X$ fixed by $h = gn$,  and let $e$ be an oriented edge with $\partial_0 e = y$. Then $e$ and $h(e)$ belong to the same $N$-orbit since $g$ acts trivially on $N\backslash X$. Since $\partial_0e = y = \partial_0 h(e)$, any element of $N$ mapping $e$ to $h(e)$ fixes $y$. Since $N$ acts freely, we deduce that $h(e) = e$. Thus $h$ fixes all  edges emanating from $y$, hence also all the neighbours of $y$. Since the graph $X$ is connected, this implies that $h = gn$ acts trivially on $X$. Thus $g \in N$ as required. 
\end{proof}

\subsection{A reduction to finite group theory}\label{sec:FiniteReduction}

The following basic result from the covering theory of graphs  allows one to go back and forth between reducible lattices in products of trees and finite groups acting on products of graphs, without affecting the local actions. 

\begin{prop}\label{prop:Infinite-Finite}
Let $T_1, T_2$ be regular trees of degree $d_1,  d_2$ and   $\Gamma \leq \Aut(T_1) \times \Aut(T_2)$ be a group acting   transitively on the vertices of the Cartesian product $T_1 \times T_2$. For $i=1, 2$, let $F_i$ denote the local action of $\Gamma$ on $T_i$,   let $K_i$ be the projection on $\Aut(T_i)$ of the kernel of the $\Gamma$-action on $T_{3-i}$. Assume that $K_i$ acts freely on $T_i$ (as defined in Section~\ref{sec:Graphs}). Then for $i=1$ and $2$, we have:
\begin{enumerate}[(i)]
 \item  the  group  $G=\Gamma/K_1\times K_2$ acts transitively on the Cartesian product $VX_1 \times VX_2$, where $X_i $ is the quotient graph $K_i \backslash T_i$,
\item $X_i$ is   of degree $d_i$ and the local action of $G$ on $X_i$ is isomorphic to $F_i$,
\item the $G$-action on $X_i$ is faithful,
\item if the $\Gamma$-action on $VT_1 \times VT_2$ is free, then so is the $G$-action on $VX_1 \times VX_2$. 
\end{enumerate}

Conversely, let   $X_1,X_2$ be regular graphs of degree $d_1,  d_2$ and   $G \leq \Aut(X_1) \times \Aut(X_2)$ be a group acting with $m$ orbits  (resp. acting regularly) on the vertices of the Cartesian product $X_1 \times X_2$. Assume that  the local action of $G$ at every vertex of $X_i$ is isomorphic to the permutation group $F_i$. Then there is a group $\Gamma \leq \Aut(T_1) \times \Aut(T_2)$ acting with $m$ orbits (resp.  acting regularly) on   $VT_1 \times VT_2$, where $T_i$ is the regular tree of degree $d_i$, such that:
\begin{enumerate}[(i)]
	\setcounter{enumi}{3}
	\item The local action of $\Gamma$ at every vertex of $T_i$ is isomorphic to $F_i$,
	\item  $\Gamma$ contains a normal subgroup of the form $K_1 \times K_2$ such that the quotient group $\Gamma/K_1 \times K_2$ is isomorphic to $G$, 
	\item $K_i$ is the fundamental group of the graph $X_i$ acting by covering transformations on the tree $T_i$.
\end{enumerate}
\end{prop}

\begin{proof}
For the first part, notice that since $K_i$ acts freely on $T_i$, the quotient graph $X_i = K_i \backslash T_i$ is   well defined. Recall that, by definition, the vertex set of $X_i$ consists of the $K_i$-orbits in $VT_i$, and the oriented edges of $X_i$ are defined as the $K_i$-orbits of oriented edges in $T_i$, so that the projection map $VT_i \to VX_i$ is a morphism of graphs. Since $K_i$ acts freely, the quotient map $T_i \to X_i$ can also be viewed as a covering map in the classical  sense from topology.  Assertions (i)--(iv) now follow from the basic covering theory of graphs (for which we refer to   \cite{Bass} and \cite{Serre}), together with Lemma~\ref{lem:Kernel}. 

The converse is also a standard application of the covering theory of graphs. 
\end{proof}

Given Proposition~\ref{prop:Infinite-Finite}, the following result is an easy consequence of known results on $s$-arc transitive Cayley graphs due to Li--Lu \cite{LiLu}, Xu--Fang--Wang--Xu \cite{XFWX} and M.~Conder~\cite{Conder}. We denote by $T_n$ the regular tree of degree $n$.

\begin{prop}\label{prop:Existence}\quad
\begin{enumerate}[(i)]
	\item There exists a reducible lattice $\Gamma_{3, 23} \leq \Aut(T_3) \times \Aut(T_{23})$ acting regularly on the vertices of   $T_3 \times T_{23}$, whose local action on $T_3$ (resp. $T_{23}$) is $\Sym(3)$ (resp. $\Sym(23)$).  

	\item There exists a reducible lattice $\Gamma_{3, 24} \leq \Aut(T_3) \times \Aut(T_{24})$ acting regularly on the vertices of   $T_3 \times T_{24}$, whose local action on $T_3$ (resp. $T_{24}$) is $\Sym(3)$ (resp. $\Sym(24)$).  

	\item  There exists a reducible lattice $\Gamma_{3, 47} \leq \Aut(T_3) \times \Aut(T_{47})$ acting regularly on the vertices of  $T_3 \times T_{47}$, whose local action on $T_3$ (resp. $T_{47}$) is $\Sym(3)$ (resp. $\Alt(47)$).
	
		\item There exists a reducible lattice $\Gamma_{4, 11663} \leq \Aut(T_4) \times \Aut(T_{11663})$ acting regularly on the vertices of   $T_4 \times T_{11663}$, whose local action on $T_4$ (resp. $T_{11663}$) is $\Sym(4)$ (resp. $\Alt(11663)$).  
	
		\item There exists a reducible lattices $\Gamma_{5, n} \leq \Aut(T_5) \times \Aut(T_{n})$ acting regularly on the vertices of   $T_5 \times T_{n}$, for $n = 19, 39$ and $79$, whose local action on $T_5$ is $\mathbf C_5 \rtimes \mathbf C_4$, and whose respective local action on $T_{19}$, $T_{39}$ and $T_{79}$ is $\Sym(19)$, $\Alt(39)$ and $\Alt(79)$. 
	
\end{enumerate}
\end{prop}
\begin{proof}
By \cite[Theorem~1.1]{LiLu}, there is a $3$-regular graph $Y$ which is a Cayley graph of the group $B = \Sym(23)$, whose full automorphism group $G$ is isomorphic to $\Sym(24)$, and such that the local action of $G$ on $Y$ is $\Sym(3)$. Let $A$ be the stabilizer in $G$ of a vertex $y \in VY$. Hence $|A| = 24$, $A \cap B = \triv$ and $G= AB$. Let moreover $X$ be the complete graph on $24$ vertices, on which $G$ acts faithfully by automorphisms. Let $x \in VX$ be the vertex fixed by $B$. Since $G=AB$ and $A \cap B = \triv$, it follows that the diagonal $G$-action on the vertex set of $X \times Y$ is free and transitive. The assertion (i) thus follows from Proposition~\ref{prop:Infinite-Finite}. 

For (ii), we use a similar argument, using a degree~$3$ Cayley graph $Y$ of an index~$2$ subgroup of $\Sym(23) \times \Sym(24)$ appearing in  \cite[Theorem~2.1(d)]{Conder}. We define $X$ to be the complete bipartite graph $\mathbf K_{24, 24}$ in this case. 

The proof of (iii), (iv) and (v) are also similar. For (iii) and (iv), one uses a degree~$3$ Cayley graph of $\Alt(47)$ appearing in \cite[Theorem~2.1(e)]{Conder} (such a graph was first constructed in \cite{XFWX}) and a degree~$4$ Cayley graph of $\Alt(11663)$ discussed in  \cite[\S3]{Conder}. For (v) and $n=39, 79$, one uses the graph from \cite{LingLu1} and \cite[Theorem~1.1(2)]{LingLu2} respectively. For (v) and $n=19$, an example was constructed by M.~Giudici using \textsc{Magma}. 
\end{proof}

\begin{rmk}\label{rem:Small}
Using the converse part of Proposition~\ref{prop:Infinite-Finite}, one can also construct reducible   lattices $\Gamma$ in regular trees of smaller degrees with $2$-transitive local actions. 

For example, consider the group $G = \Sym(4) \times \mathbf{C}_2$. It acts locally $2$-transitively on the bipartite graph $X$ with $2$ vertices and $4$ geometric edges, as well as  on the graph $Y$ which is the $1$-skeleton of the cube. The diagonal action of $G$ on $X \times Y$ has $2$ orbits of vertices, and the vertex-stabilizers are non-trivial (they are isomorphic to $\Sym(3)$). Invoking Proposition~\ref{prop:Infinite-Finite}, we obtain a locally $2$-transitive reducible lattice $\Gamma \leq \Aut(T_4) \times \Aut(T_3)$ acting with $2$~orbits of vertices, since $\widetilde X \cong T_4$ and $\widetilde Y \cong T_3$. 

Another example is constructed similarly using  the group $G = \Sym(5)$, that  acts locally $2$-transitively on the complete graph $X= \mathbf K_5$, as well as on the Petersen graph $Y$. The diagonal action of $G$ on $X \times Y$ has $2$~orbits of vertices (the  stabilizers of vertices in the corresponding orbits   are respectively of order~$4$ and~$6$). This yields   a locally $2$-transitive reducible lattice $\Gamma \leq \Aut(T_4) \times \Aut(T_3)$ acting with $2$~orbits of vertices, since $\widetilde{\mathbf K_5} \cong T_4$ and $\widetilde Y \cong T_3$. 
\end{rmk}

\subsection{Locally $2$-transitive actions}

Recall that a permutation group $G \leq \Sym(\Omega)$ is \textbf{quasi-primitive} if every non-trivial normal subgroup of $G$ acts transitively on $\Omega$. 

\begin{lem}[{\cite[Lemma~1.4.2]{BuMo1}}] \label{lem:BuMo:LocallyQuasiPrimitive}
Let $X$ be a connected graph, let $G \leq \Aut(X)$ be a group whose local action is quasi-primitive and let $N \leq G$ be a normal subgroup of $G$. Set 
$$\begin{array}{rcl}
VX'(N) &= &\{x \in VX \mid N_x \text{ acts transitively on } E(x)\},\\
VX''(N) &= &\{x \in VX \mid N_x \text{ acts trivially on } E(x)\}.\\
\end{array}$$
Then one of the following assertions holds:
\begin{enumerate}[(i)]
	\item $VX''(N) = X$ and $N$ acts freely on $VX$. 
	\item $VX'(N) = X$ and $N$ acts transitively on the set of geometric edges of $X$. In particular $N$ has at most $2$~orbits of vertices. 
	\item $VX = VX'(N) \cup VX''(N)$ is a $G$-invariant bipartition of $X$, and for any $x'' \in VX''(N)$, the $1$-ball $B(x'', 1)$ around $x''$ is a strict fundamental domain for the $N$-action on $X$. 
\end{enumerate}
\end{lem}

The case (ii) splits into two subcases, according to whether $N$ is transitive on $VX$. In particular, we deduce the following when $G$ is vertex-transitive. 

\begin{cor}\label{cor:BuMo}
	Let $X$ be a connected graph, let $G \leq \Aut(X)$ be a vertex-transitive group whose local action is quasi-primitive. For any normal subgroup  $N \leq G$, one of the following assertions holds:
	\begin{enumerate}[(i)]
		\item $N$ acts freely on $VX$. 
		
		\item $N$ is transitive on $VX$ and  on the set of geometric edges. 
		
		\item $N$ has exactly two orbits on $VX$, which form a $G$-invariant bipartition of $X$, and $N$ is transitive on the set of geometric edges. 

	\end{enumerate}

\end{cor}	
\begin{proof}
Since $N$ is normal and $G$ is vertex-transitive, the $N_x$-action on $E(x)$ is isomorphic to the $N_y$-action on $E(y)$ for any two vertices $x, y \in VX$. Thus only the cases (i) or (ii) from Lemma~\ref{lem:BuMo:LocallyQuasiPrimitive} can occur. In the second case, observe that if $N$ is not transitive on $VX$, then no element of $N$ can map a vertex to a neighbour, because $ N_x \text{ acts transitively on } E(x)$ for all $x \in VX$. Thus the $N$-orbits form a $G$-invariant partition of $VX$ such that no two element of a given class are adjacent. Since $N$ is transitive on the set of geometric edges, it has at most $2$~orbits of vertices. The desired assertion follows.  
\end{proof}

In case  $N \leq G$ is a normal subgroup acting non-freely on $VX$, we have the following. 

\begin{cor}\label{cor:BuMo:2}
	Let $X$ be a connected   graph, let $G \leq \Aut(X)$ be a vertex-transitive group whose local action is quasi-primitive.  Let       $N \leq G$ be a normal subgroup. Assume that neither $N$ nor $\mathrm{C}_G(N)$ acts  freely on $VX$. Then either  $|VX | \leq 2$, or $X$ is complete bipartite and $N$ acts regularly on the set of geometric edges. 
\end{cor}	
\begin{proof}
	Let $M = \mathrm{C}_G(N)$. Since $N$ is normal in $G$, so is $M$. In view of the hypotheses, both $M$ and $N$ satisfy the conclusion (ii) in Lemma~\ref{lem:BuMo:LocallyQuasiPrimitive}. 
	
	We now invoke Corollary~\ref{cor:BuMo} for $M$. 
	
	If $M$ is transitive on $VX$, then $N_{x} = N_y$ for any pair of vertices $x, y \in VX$. Since $N_x$ is also transitive on $E(x)$ it follows that $|VX| \leq 2$. 
	
If $M$ is not transitive on $VX$, then $X$ is bipartite and $M$ has two orbits on $VX$. Let $x\neq y$ be adjacent vertices. 


The $M_y$-action on $E(y)$ is transitive by Lemma~\ref{lem:BuMo:LocallyQuasiPrimitive}. Thus, for all neighbours $x'$ of $y$, we have  $N_x = N_{x'}$. Since $N_x$  is transitive on $E(x)$, we see that $N_x$-orbit of $y$ coincides with the set of neighbours of $x$. Since $N_x = N_{x'}$ and  $N_{x'}$  is transitive on $E(x')$, we deduce that $x$ and $x'$ have the same set of neighbours. Similarly, any neighbour $y'$ of $x$ has the same set of neighbours as $y$. Since $X$ is connected, this implies 
that $X$ is a complete bipartite graph. Given a geometric edge $ \{x, y\}$, the stabilizer $N_{x, y}$ is trivial since it commutes with $M$, which is transitive on the geometric edges by Corollary~\ref{cor:BuMo}. 
The conclusion follows since $N$ is transitive on the set of geometric edges by Corollary~\ref{cor:BuMo}. 
\end{proof}

%

We also record the following information about the case where the local action of $G$ is $2$-transitive and $N \leq G$ is a normal subgroup acting freely on $VX$ but not freely on geometric edges. 

\begin{lem}\label{lem:Nfree}
Let $X$ be a connected graph, let $G \leq \Aut(X)$ be a group whose local action is $2$-transitive,  and let $N \leq G$ be a normal subgroup of $G$ acting freely on $VX$ but non-freely on the set of geometric edges of $X$. Let $x \in VX$. Then:
\begin{enumerate}[(i)]
	\item For each $e \in E(x)$, there is a unique element $s_e \in N$ with $s_e(e) = \bar e$ and $s_e^2 = 1$. 
	\item $N$ acts regularly on $VX$. 
	\item $N$ is generated by the set $\{s_e \mid e \in E(x)\}$. 
\item $G_x^{[1]} = \{1\}$. 
\item $G \cong N \rtimes G_x$ and $\mathrm{C}_{G_x}(N)=\{1\}$. 

\item $\mathrm Z(G) \leq N$.
\end{enumerate}
\end{lem}

\begin{proof}
The hypotheses on $N$ imply the existence of an edge $f \in EX$ and an element $s \in N$ with $s(f)=\bar f$. Since $N$ is free on $VX$ we have $s^2=1$. Let $x = \partial_0(f)$. For each $e \in E(x)$ there is $g \in G_x$ with $g(f)=e$. Set $s_e = gsg^{-1} \in N$. Thus we have proved Assertion~(i) for some vertex $x$, and the assertion will follow for all vertices as soon as we show that $N$ is vertex-transitive. The group $\langle s_e \mid e \in E(x)\rangle$ contains an element mapping $x$ to each of its neighbours. Since $X$ is connected, it follows that the latter group is transitive on $VX$. Thus $N$ is transitive and Assertions~(i), (ii) and~(iii) follow since $N$ acts freely on $VX$ by hypothesis. Moreover (v) is a consequence of (ii). Finally, observe that an element  $g \in G_x^{[1]}$ fixes each  $e \in E(x)$, and thus centralizes $s_e$. Thus $g \in \mathrm{C}_{G_x}(N)$ by (iii).  Thus $g=1$ by (v), and (iv) holds. 

Let $Z = \mathrm Z(G)$ be the center of $G$. Its image under the projection $G \to G/N \cong G_x$ is a central subgroup of $G_x$. The group $G_x$ acts $2$-transitively on $E(x)$, and that action is faithful by (iv). It follows that $\mathrm Z(G_x)=\{1\}$. Hence $Z \leq N$ and (vi) holds. 
\end{proof}

\subsection{Vertex stabilizers of locally $2$-transitive actions}

The following important result due to V.~Trofimov and R.~Weiss provides  very precise information about vertex-strabilizers for proper  vertex-transitive locally $2$-transitive actions of discrete groups on locally finite graphs. It plays a crucial role in our considerations. 

\begin{thm}\label{thm:TrofimovWeiss}
Let $G \leq \Aut(X)$ be a vertex-transitive automorphism group of a connected locally finite graph  $X$. Let $(v, w)$ be an edge of $X$. Suppose that the local action  is $2$-transitive, and that  the stabilizer $G_v$ is finite. Then:
\begin{enumerate}[(i)]
	\item {\upshape (Trofimov--Weiss \cite[Theorem~1.4]{TrofimovWeiss})} We have
	$$G_v^{[5]}  \cap G_w^{[5]}= \{1\}.$$ 
	In particular  $G_v^{[6]} = \{1\}$. 
	
	\item {\upshape (Trofimov--Weiss \cite[Theorem~1.3 and~2.3]{TrofimovWeiss})} If   $G_v^{[1]}  \cap G_w^{[1]}  \neq  \{1\}$ (e.g. if $G_v^{[2]} \neq \{1\}$), then the local action at $v$ contains a normal subgroup isomorphic to $\PSL_n(\FF_q)$ in its natural action on the points of the $n-1$-dimensional projective space over the finite field $\FF_q$ of order~$q$. Moreover   $G_v^{[1]}  \cap G_w^{[1]}$ is a $p$-group, where $p$ is the characteristic of $\FF_q$.

		\item  {\upshape (R.~Weiss \cite[Theorem~1.1 and~1.4]{Weiss79})} If  the local action at $v$ contains a normal subgroup isomorphic to $\PSL_2(\FF_q)$ in its natural action on the points of the projective line   over a finite field $\FF_q$, then there is $s \in \{2, 3, 4, 5, 7\}$ such that for any geodesic segment $(v_1, v_2, \dots, v_s)$ of length $s-1$, we have
		$$G_{v_1}^{[1]} \cap G_{v_2}^{[1]}  \cap G_{v_3} \cap \dots \cap G_{v_s} = \triv.$$
Moreover if $\mathrm{char}(\FF_q) \geq 5$ then $s \leq 4$, and  if $\mathrm{char}(\FF_q) = 2$ then $s \leq 5$.  		

\end{enumerate}

\end{thm}

\subsection{The $2$-transitive groups of degree~$\leq 6$}

In the proof of Theorem~\ref{thm:Main}, we will encounter several case-by-case discussions depending notably on the list of $2$-transitive groups of small degree. For the reader's convenience, that list is recalled in Table~\ref{tab:Small2Trans}. 

\begin{table}[h]
$$\begin{array}{c@{\hspace{.9cm}}c@{\hspace{.9cm}}c}
\text{Degree } d &  G \leq \Sym(d) &  |G|\\
\hline
3 & \Sym(3)  \cong \mathbf{C}_3 \rtimes \mathbf{C}_2  \cong \FF_3\rtimes \FF_3^* & 6\\
4 & \Alt(4) \cong  \PSL_2(\FF_3)  \cong \FF_4\rtimes \FF_4^*  & 12\\
4 & \Sym(4) \cong \PGL_2(\FF_3) & 24\\
5 & \mathbf{C}_5 \rtimes \mathbf{C}_4 \cong \FF_5\rtimes \FF_5^* & 20\\
5& \Alt(5) \cong \PSL_2(\FF_4) & 60\\
5 & \Sym(5) \cong \PGammaL_2(\FF_4) & 120\\
6 & \Alt(5) \cong \PSL_2(\FF_5) & 60\\
6 & \Sym(5) \cong \PGL_2(\FF_5) & 120\\
6 & \Alt(6) & 360\\
6 & \Sym(6) & 720
\end{array}$$
\caption{$2$-transitive groups of degree~$\leq 6$} \label{tab:Small2Trans}
\end{table}

Keeping that list in mind, we now present two consequences of  Theorem~\ref{thm:TrofimovWeiss} needed for the proof of Theorem~\ref{thm:Main}.  The following one should be compared with   \cite[Lemma 3.5.1]{BuMo1}.

\begin{cor}\label{cor:AlmostSimplePointStab}
	Let $G \leq \Aut(X)$ be a vertex-transitive automorphism group of a connected locally finite graph  $X$ of degree $d$ with finite vertex-stabilizers. Suppose that the local action $F \leq \Sym(d)$ is $2$-transitive. Suppose moreover that at least one of the following conditions holds:
	\begin{enumerate}[(a)]
		\item  the point stabilizer $F_p$ is almost simple. 
		\item $F$ is sharply $2$-transitive  and $d \leq 5$. 
	\end{enumerate}
	Then for $v \in VX$, we have $G_v^{[2]} = \triv$, and the group   $G_v^{[1]}$ is isomorphic to a normal subgroup of $F_p$. Furthermore,   if (a) holds and  if   $G_v^{[1]}\neq \triv$ then   $G_v^{[1]}$ is  almost simple with socle isomorphic to $\soc(F_p)$.
\end{cor}
\begin{proof}
	Let $w \in VX$ be adjacent of $v$. Each of the conditions (a) and (b)  implies that $G_v^{[1]}  \cap G_w^{[1]}  = \triv$  by Theorem~\ref{thm:TrofimovWeiss}(ii) (see Table~\ref{tab:Small2Trans}). The groups $G_v^{[1]} $ and $G_w^{[1]}$ are both normal subgroups of $G_{v, w}$, and the quotient  $G_{v, w}/ G_w^{[1]}$ is isomorphic to $F_p$. The image of  $G_v^{[1]} \leq   G_{v, w}$ under the projection $G_{v, w} \to G_{v, w}/ G_w^{[1]}$ is injective (since $G_v^{[1]}  \cap G_w^{[1]}  = \triv$) and isomorphic to a normal subgroup of the  group $F_p$. The required conclusions follow. 
\end{proof}

The various possible exceptions appearing in Theorem~\ref{thm:Main} find their roots in the following result. 

\begin{cor}\label{cor:TroWeiss}
	Let $G \leq \Aut(X)$ be a vertex-transitive automorphism group of a connected locally finite graph  $X$ of degree $d$ with finite vertex-stabilizers. Suppose that the local action $F \leq \Sym(d)$ is $2$-transitive, and moreover that $F \geq \Alt(d)$ if $d \geq 7$.  Let   $x \in VX$. Then one of the following assertions holds:
	\begin{enumerate}[(i)]
		\item $d \geq 6$ and $|G_x| \in \{\frac{d!} 2, d!, \frac{d!(d-1)!} 4,  \frac{d!(d-1)!} 2,  {d!(d-1)!} \}$. 
		
		\item $d=3$ and $|G_x| \in \{6n  \mid  n \text{ divides } 2^3\}$. 
		
		\item $d=4$ and $|G_x| \in \{12n \mid n \text{ divides } 2^2 \cdot 3^5\}$. 

		\item $d=5$ and $F = \mathbf{C}_5 \rtimes \mathbf{C}_4$, then $|G_x| \in \{20, 40, 80\}$. 
				
		\item $d=5$ and $\soc(F) = \Alt(5)$ and $|G_x| \in \{60n \mid n \text{ divides } 2^{8} \cdot 3\}$. 
		
		\item $d=6$, $\soc(G) = \PSL_2(\FF_5)$ and $|G_x| \in \{60n \mid n \text{ divides } 2^{3} \cdot 5^2\}$. 
	\end{enumerate}
\end{cor}
\begin{proof}
If $F = \Alt(d)$ or $\Sym(d)$ with $d \geq 6$, we must have (i) by Corollary~\ref{cor:AlmostSimplePointStab}. Similarly, if $d=5$ and $F = \mathbf{C}_5 \rtimes \mathbf{C}_4$ then  $|G_x| \in \{20, 40, 80\}$ by Corollary~\ref{cor:AlmostSimplePointStab}. In the remaining cases,   we apply Theorem~\ref{thm:TrofimovWeiss}(iii) using the list in Table~\ref{tab:Small2Trans}. 
\end{proof}

\begin{rmk}\label{rem:sharpen}
The structure of $G_x$ in the case where $d \leq 6$ can be described more precisely, see Theorems~(1.2) and~(1.3) in \cite{Weiss79}. Those results could be used to sharpen slightly the range of values appearing in Corollary~\ref{cor:TroWeiss}, and hence also those in Theorem~\ref{thm:Main}; we will not perform that sharpening here. 
\end{rmk}

%
%
%

\subsection{Finite simple $\{2, 3, 5\}$-groups} 

The following result is a consequence of the CFSG. 
\begin{prop}[{\cite[Theorem~III(1) and Table~1]{HuppertLempken}}]\label{prop:HL}
Let $S$ be a non-abelian finite simple group such that the only prime divisors of $|S|$ are $2, 3$ and $5$. Then
$S$ is isomorphic to $\Alt(5)$, $\Alt(6)$ or $\mathrm{PSp}_4(3) \cong U_4(2)$, respectively of order $2^2 \cdot 3 \cdot 5$, $2^3 \cdot 3^2 \cdot 5$ and $2^6 \cdot 3^4 \cdot 5$. 
\end{prop}

\subsection{Subgroups of a finite simple group involving all its primes}\label{sec:LPS}

Given a finite set $X$, we denote by $\pi(X)$ the set of prime divisors of $|X|$. The following important  result will be crucial to our purposes. 

\begin{thm}[{Liebeck--Praeger--Saxl \cite[Corollary~5]{LPS}}] \label{thm:LPS}
Let $G$ be a finite almost simple group with socle $N$. Let $M \leq G$ be a subgroup not containing $N$ such that $\pi(M) \supseteq \pi(N)$. Then the possibilities for $N$ and $M$ are all listed in \cite[Table~10.7]{LPS}.
\end{thm}

The following consequence, that can be extracted from the list given by Liebeck--Praeger--Saxl, will be sufficient for us.  (Extra caution is needed in view of the exceptional isomorphisms between small finite simple groups.) 

\begin{cor}\label{cor:LPS}
Retain the assumptions of Theorem~\ref{thm:LPS} and suppose in addition that $M \cap N$ has a composition factor isomorphic to $\Alt(d)$ for some $d \geq 5$. Then either there exist positive integers $k \leq c$ such that $N = \Alt(c)$ and $\Alt(k) \lhd M \leq \Sym(k) \times \Sym(c-k)$ and $k \geq p$ for all primes $p \leq c$, or the pair $(N, M \cap N)$ is one of the exceptions listed in Table~\ref{tab:LPS} (see \cite{LPS} for the notation). 
\end{cor}

\begin{table}[h]
$$
\begin{array}{r@{\hspace{.9cm}} c @{\hspace{.9cm}}c @{\hspace{.9cm}} c}
& N & |N| & M \cap N  \\
\hline 
(1) & \Alt(6) & 2^3 \cdot 3^2 \cdot 5 &  L_2(5) \cong \Alt(5)\\
(2) & \mathrm U_3(5) & 2^4 \cdot 3^2 \cdot 5^3 \cdot 7 & \Alt(7)\\
(3) & \mathrm U_4(2)  & 2^6 \cdot 3^4 \cdot 5 & M \cap N \leq 2^4.\Alt(5),   \Sym(6)\\
(4) &  \mathrm U_4(3) & 2^7 \cdot 3^6\cdot  5 \cdot 7 & \Alt(7)\\
(5) &  \mathrm{PSp}_4(7) & 2^8 \cdot 3^2 \cdot  5^2 \cdot 7^4 & \Alt(7)\\
(6) & \mathrm{Sp}_6(2) & 2^9 \cdot 3^4 \cdot 5 \cdot 7 & \Alt(7), \Sym(7), \Alt(8), \Sym(8)\\
(7) & \mathrm P\Omega^+_8(2) & 2^{12} \cdot 3^5 \cdot 5^2 \cdot 7 & M \cap N \leq P_1, P_3, P_4, \Alt(9)\\
\end{array}
$$
\caption{The exceptional pairs $(N, M\cap N)$ in Corollary~\ref{cor:LPS}} \label{tab:LPS}
\end{table}


\subsection{On subgroups of direct products of simple groups}

The following  fact is an easy corollary of an important consequence of the Classification of the Finite Simple Groups, namely the fact the every automorphism of a non-abelian finite simple group centralizes a non-trivial subgroup (see \cite[\S9.5.3]{KurSte}). Note however that we shall  need that result only in the case of the alternating groups.

\begin{prop}\label{prop:Goursat}
	Let $S$ be a non-abelian finite simple group.  Let $G = S_1 \times S_2$  be the direct product of two groups  isomorphic to $S$,  and  let $A_1, A_2 \leq G$ be subgroups of $G$  that are also  isomorphic to $S$. If $A_1 \cap A_2 = \triv$, then there is $i \in \{1, 2\}$ such that  $A_i = S_1 \times \triv$ or $A_i = \triv \times S_2$.  
\end{prop}
\begin{proof}
Assume that $A_i \neq S_1 \times \triv$ and $A_i \neq \triv \times S_2$ for $i = 1$ and $2$. Then by Goursat's Lemma, for $i=1, 2$ there exists an isomorphism $\varphi_i \colon S_1 \to S_2$ such that $A_i = \{(x, \varphi_i(x)) \mid x \in S_1)\}$. Since $A_1 \cap A_2 = \triv$, it follows that $\varphi_1^{-1} \varphi_2$ is an automorphism of $S_1$ whose only fixed point is the trivial element. By \cite[\S9.5.3]{KurSte} (see also the announcement in \cite[Theorem 1.48]{Gorenstein1982}), the group $S_1$ must be solvable, contradicting the hypotheses. 
\end{proof}

\section{Finite groups with locally $2$-transitive actions on product graphs}

The goal of this section is to prove Theorem~\ref{thm:Main}. It will be deduced as a corollary to the following result. See Section~\ref{sec:proofs}.  

\begin{thm}\label{thm:MainFinite}
	Let $X_1, X_2$ be finite connected regular graphs of degree $d_1 \geq d_2 \geq 3$ and let $G \leq \Aut(X_1) \times \Aut(X_2)$ be a group acting freely and  transitively on the vertices of the Cartesian product $X_1 \times X_2$.  For $i=1$ and $2$, we assume that the $G$-action on $X_i$ is faithful and that its local action $F_i$ is locally $2$-transitive; we assume moreover that if $d_i \geq 7$ then   $F_i \geq \Alt(d_i)$. Then $X_1$ is the complete graph $\mathbf K_{d_1+1}$ or the complete bipartite graph $\mathbf K_{d_1, d_1}$. Moreover  one of the conditions (i)--(vi) listed in Theorem~\ref{thm:Main} is satisfied. 
\end{thm} 

The proof occupies the rest of this section.

\subsection{The standing hypotheses and notation}\label{sec:not}

We   fix the notation and assumptions adopted throughout.  
For  $i =1, 2$, let  $d_i \geq 3$ and   $F_i \leq \Sym(d_i)$ be a $2$-transitive permutation group.  
Let $\mathcal E(F_1, F_2)$ be the collection of triples $(X_1, X_2, G)$ satisfying the following conditions: 
\begin{description}
	
\item[(Hyp1)]     $X_i$ is a connected $d_i$-regular graph for $i=1$ and $2$.

\item[(Hyp2)]  $G \leq \Aut(X_1) \times \Aut(X_2)$ is a finite group. 

\item[(Hyp3)] $G$ acts  transitively on $VX_1 \times VX_2$.

\item[(Hyp4)] The $G$-action on $X_i$ is faithful for $i=1$ and $2$. 

\item[(Hyp5)] The local action of $G$ on $X_i$ is isomorphic to $F_i$ for $i=1$ and $2$. 
\end{description}
 
We further denote by $\mathcal F(F_1, F_2)$ the subcollection consisting of those triples $(X_1, X_2, G) \in \mathcal E(F_1, F_2)$ satisfying in addition:
\begin{description}

\item[(Hyp6)] $G$ acts  freely on $VX_1 \times VX_2$. 
\end{description}

\begin{lem}\label{lem:Basic}
Let $(X_1, X_2, G) \in \mathcal E(F_1, F_2)$, let $x_1 \in VX_1$ and $x_2 \in VX_2$. Then:
\begin{enumerate}[(i)]
	\item $G_{x_1}$ acts   transitively on $VX_2$ and $G_{x_2}$ acts  transitively on $VX_1$. 
	
	\item $G = G_{x_1} G_{x_2}$. 
	
\end{enumerate}
If in addition $(X_1, X_2, G) \in \mathcal F(F_1, F_2)$, then:
\begin{enumerate}[(i)]
	\setcounter{enumi}{2}
	\item $G_{x_1}$ acts   freely on $VX_2$ and $G_{x_2}$ acts  freely on $VX_1$. 
	
	\item $G_{x_1}  \cap G_{x_2} = \{1\}$. 
	
\end{enumerate}

\end{lem}
\begin{proof}
Assertion (i) is immediate from (Hyp2) and (Hyp3); assertion (ii) follows from (i), while (iii) and (iv) are equally straightforward. 
\end{proof}

Thus, if $(X_1, X_2, G) \in \mathcal F(F_1, F_2)$, we may  see $X_1$ as a Cayley graph of $G_{x_2}$ and vice-versa, unless $|VX_1| \leq 2$ (resp. $|VX_2|\leq 2$). The latter inequality is never satisfied in our setting: indeed, since $|VX_i| = |G_{x_{3-i}}|$ and since $G_{x_{3-i}}$ has a $2$-transitive action on a set of cardinality $d_{3-i} \geq 3$, we have $|VX_i| \geq 3$. 

\subsection{Proof of Theorem~\ref{thm:NonExplicit}}

As mentioned in the introduction, Theorem~\ref{thm:NonExplicit} is a straightforward consequence of Theorem~\ref{thm:TrofimovWeiss}. Let us  record the proof. We first need the following result. 

\begin{lem}\label{lem:no-inversion}
Let  $T_1, T_2$ be regular trees of degree $d_1, d_2 \geq 3$ and let $\Gamma \leq \Aut(T_1) \times \Aut(T_2)$ be a group acting regularly on the vertices of $T_1 \times T_2$. Assume that the local action of $\Gamma$ on $T_i$ is $2$-transitive for $i=1, 2$. For $i=1, 2$, let    $K_i$ be the projection on $\Aut(T_i)$ of the kernel of the $\Gamma$-action on $T_{3-i}$. Then $K_i$ acts freely on $T_i$. 
\end{lem}
\begin{proof}
	Since $\Gamma$ acts freely on $VT_1 \times VT_2$, it follows that $K_i$ acts freely on $VT_i$. We need to show that $K_i$ does not invert any edge of $T_i$. If $K_i$ contains an edge inversion, then $K_i$ acts regularly on $VT_i$ by Lemma~\ref{lem:Nfree}(ii). Let $v  \in VT_{3-i}$. Let $K'_i$ be the kernel of the $\Gamma$-action on $T_{3-i}$, so that $K'_i \cong K_i$ and $K'_i$ acts regularly on $VT_i$. Clearly $K'_i \leq \Gamma_v$.   Since $\Gamma$ acts regularly on $VT_1 \times VT_2$, it follows that $\Gamma_v$ is regular on $VT_i$, so that $K'_i  = \Gamma_v$. Since that equality holds for all  $v \in VT_{3-i}$, it follows that $\Gamma_v$ acts trivially on $T_{3-i}$. This contradicts the hypothesis that $\Gamma_v$ is $2$-transitive on $E(v)$. 
\end{proof}

The example following Theorem~\ref{thm:NonExplicit} in the introduction shows that Lemma~\ref{lem:no-inversion} may fail if the $\Gamma$-action on $VT_1 \times VT_2$ is not free. 

\begin{proof}[Proof of Theorem~\ref{thm:NonExplicit}]
Retain the notation of Lemma~\ref{lem:no-inversion} and assume that $d_1 \geq d_2$ and that  $\Gamma$ is reducible. We must show that $d_1 <M$, where 
$$M =  (d_2!)\big((d_2-1)!\big)^{\frac{d_2 ((d_2-1)^5-1 )}{d_2-2}}.$$

The reducibility of $\Gamma$ ensures that the quotient $\Gamma/K_1 \times K_2$ is finite. Moreover by Lemma~\ref{lem:no-inversion}, we may invoke Proposition~\ref{prop:Infinite-Finite}, which ensures that the set $\mathcal F(F_1, F_2)$ is non-empty, where $F_1, F_2$ denote the local actions of $\Gamma$ on $T_1,T_2$.  Let $(X_1, X_2, G) \in \mathcal F(F_1, F_2)$, let $x \in VX_2$ and  $v \in VT_2$. In view of Theorem~\ref{thm:TrofimovWeiss}, an upper bound on the order of $|G_{x}|$ is provided by the order $\Aut(T_2)_{v}/\Aut(T_2)_{v}^{[6]}$. The latter group is isomorphic to the iterated permutational wreath product
$$\Sym(d_2-1) \wr \Sym(d_2-1) \wr \Sym(d_2-1) \wr \Sym(d_2-1) \wr \Sym(d_2-1)\wr \Sym(d_2),$$
whose order is $M$.
In particular $X_1$ is a $d_1$-regular graph of order bounded above by that number. Since $\Aut(X_1)$ is locally $2$-transitive, we have $|VX_1| \leq 2$ or $|VX_1| \geq d_1+1$. The former case is impossible, since it would imply that $|G_{x_2}| = 2$ by Lemma~\ref{lem:Basic}, contradicting that $G$ is locally $2$-transitive on the graph  $X_2$ whose  degree is $d_2 \geq 3$. Thus we obtain $d_1 +1 \leq M$,  which is the required bound. 
\end{proof}

\begin{rmk}
The bound obtained in the proof above can directly be sharpened by exploiting  Theorem~\ref{thm:TrofimovWeiss} in a more precise way.  We will do this   in the proof of Theorem~\ref{thm:MainFinite}. 
\end{rmk}

\subsection{Assuming that $N$ acts freely on $VX_1$ and on $VX_2$}

From now on, we choose a member $(X_1, X_2, G) \in \mathcal F(F_1, F_2)$. 
We also fix $N \neq \{1\}$  a minimal normal subgroup of $G$. Thus $N$ is characteristically simple. Hence it is   isomorphic to  the $k$-th direct power of a finite simple group $S$. We also fix $x_1 \in VX_1$ and $x_2 \in VX_2$.

\begin{lem}\label{lem:DoublyFree}
	Assume that $N \cong S^k$ acts freely on both  $VX_1$ and $VX_2$,  but not freely on the set of geometric edges of $X_i$ for $i = 1$ or $2$.  Then:
	\begin{enumerate}[(i)]
		\item $|N| = |G_{x_{3-i}}|$. 
		\item $G_{x_{3-i}}$ is isomorphic to a subgroup of $G_{x_i}$. 
		\item $N$ and $S$ are not abelian.
		\item $d_1, d_2 \geq 5$. 
		\item $|F_i|$ has at least~$3$ prime divisors. 
		
	\end{enumerate} 
\end{lem}
\begin{proof}
	Lemma~\ref{lem:Nfree} applies to the $N$-action on  $X_i$. It follows that $G \cong N \rtimes G_{x_i}$. In view of Lemma~\ref{lem:Basic}, the assertion (i) follows. Since the $N$-action on $VX_{3-i}$ is free, the projection of $G_{x_{3-i}}$ under the  projection $G \to G/N  \cong   G_{x_i}$ maps $G_{x_{3-i}}$ injectively onto a subgroup of $G_{x_i}$. This proves (ii). 
	
	If $N \cong S^k$ were abelian (or equivalently if $S$ were abelian), then the order of $N$ would be a power of~$2$ since $N$ is generated by involutions (see Lemma~\ref{lem:Nfree}(i) and (iii)). Thus $G_{x_{3-i}}$ is a $2$-group by (i). A $2$-group does not admit a $2$-transitive action on a set containing more than two elements. This is a contradiction since $G_{x_{3-i}}$  is $2$-transitive on a set of cardinality $d_{3-i} \geq 3$.  This proves (iii). If $d_1 \leq 4$ or $d_2 \leq 4$, then the set of prime divisors of $|G_{x_1}|$ or  $|G_{x_2}|$ would be contained in $\{2, 3\}$. Hence the same would apply to $|N|$   by (i) and (ii). Thus $N$ would be solvable by Burnside's theorem, hence abelian since $N$ is characteristically simple. This contradicts (iii). Thus (iv) holds. If $|F_i|$ has at most~$2$ prime divisors, then the same holds for $|G_{x_i}|$, hence also $|G_{x_{3-i}}|$ by (ii), hence $|G|$ by Lemma~\ref{lem:Basic}. Thus $G$ is solvable by Burnside' theorem. The minimal normal subgroup $N$ must thus be abelian,  contradicting (iii). This proves (v). 
%
%
\end{proof}

\begin{lem}\label{lem:DoublyFree:2cases}
	For $j = 1$ and $2$, we assume that $F_j \geq \Alt(d_j)$ if $d_j \geq 7$. 
If $N$ acts freely on    $VX_1$ and on $VX_2$,  then it acts freely on $X_1$ and on $X_2$. 
\end{lem}
\begin{proof}
	Suppose for a contradiction that $N$ does not act freely on $X_i$ for some $i \in \{1, 2\}$. This means that $N$ does not act freely on the set of geometric edges of $X_i$. By Lemma~\ref{lem:DoublyFree}(iv), we have $d_i \geq 5$. Moreover $F_i$ has at least~$3$ distinct prime divisors by Lemma~\ref{lem:DoublyFree}(v). In particular $F_i \not \cong \mathbf C_5 \rtimes \mathbf C_4$. By the hypothesis made on $F_i$, we deduce that $F_i$ is not solvable (see Table~\ref{tab:Small2Trans}). The hypotheses imply that the socle of $F_i$ is isomorphic to $\Alt(d_i)$, or to $\Alt(5)$ if  $d_i = 6$. Note moreover that, since $G_{x_i}^{[1]}= \triv$ by Lemma~\ref{lem:Nfree}(iv), we have $G_{x_i} \cong F_i$.  
	
	Recall that $N \cong S^k$, where $S$ is a finite simple group. We distinguish two cases. 

Suppose first that $C_G(N) = \{1\}$. Then the $G$-conjugation action on $N$ yields an injective homomorphism of $G_{x_i} \cong G/N$ into $\Out(N)$. By the Krull--Remak--Schmidt theorem (see \cite[Theorem 3.3.8]{Robinson}), the outer automorphism group $\Out(N)$ is isomorphic to the wreath product $\Out(S) \wr \Sym(k)$. The group $\Out(S)$ is solvable by the Schreier conjecture. Since $G_{x_i}$ is not solvable, we deduce that $k \geq 5$. 
	
	Since $|N| = |G_{x_{3-i}}|$ by Lemma~\ref{lem:DoublyFree}(i), we infer that the order of $G_{x_{3-i}}$ is a $k^{\mathrm{th}}$ power. We have $d_{3-i} \geq 5$ by   Lemma~\ref{lem:DoublyFree}(iv). Let $p$ be a prime with $		d_{3-i} /2 < p < d_{3-i} $. Using Corollary~\ref{cor:TroWeiss}, we see that $p$ divides $|G_{x_{3-i}}|$, but $p^4$ does not.  Therefore  $|G_{x_{3-i}}|$ cannot be a $k^{\mathrm{th}}$ power with $k \geq 4$, and we have reached a contradiction. This finishes the proof in the case where  $C_G(N) = \triv$. 
	
	Assume now that $C_G(N)$ is non-trivial. Let $M \neq \{1\}$ be a minimal normal subgroup of $G$ contained in $C_G(N)$. Thus $MN$ is a normal subgroup of $G$. Since $N$ is minimal normal in $G$ and non-abelian by Lemma~\ref{lem:DoublyFree}(iii), we have $M \cap N = \{1\}$. In view of Lemma~\ref{lem:Nfree}(v), we deduce that $MN \cap G_{x_i} \neq \triv$. Since $G_{x_i} \cong F_i$ is almost simple with socle $\Alt(d_i)$ or $\Alt(5)$ if $d_i =6$, we deduce that  $G^+_{x_i} := G_{x_i} \cap MN$ contains the socle of $G_{x_i}$, and has index $1$ or $2$ in $G_{x_i}$. Lemma~\ref{lem:Nfree}(v) also implies that $M \cap G_{x_i} = \triv$, so that the natural homomorphism $MN \to N$ yields an injective homomorphism $G^+_{x_i} \to N$. By Lemma~\ref{lem:DoublyFree}(i) and (ii), we also have $|N| \leq |G_{x_i}|$, hence $|N| \leq 2 |G_{x_i}^+|$. It follows $G^+_{x_i}$ is isomorphic to a subgroup of index at most~$2$ in $N$. Since $N$ is characteristically simple and non-abelian, we deduce that $G^+_{x_i} \cong N$. In particular $N$, hence also $G^+_{x_i}$, is simple. This implies that $G^+_{x_i}$ is the socle of $G_{x_i}$. 
	
Since $G \cong N \rtimes G_{x_i}$ and $M \cap N = \{1\}$, the projection map $G \to G_{x_i}$ yields an injective homomorphism of  $M$ into $G_{x_i}$, whose image is normal since $M$ is normal in $G$. Since $G_{x_i}$ is almost simple with socle $G^+_{x_i}$, while $M$ is characteristically simple, it follows that $G^+_{x_i}  \cong M$. Using Lemma~\ref{lem:DoublyFree}(i) and (ii), we obtain that $G_{x_{3-i}} \cong M \cong N \cong G^+_{x_i}$, which is $\Alt(d_i)$ or $\Alt(5)$. Since $MN$ has index at most~$2$ in $G$ by Corollary~\ref{cor:BuMo}, we have $G_{x_{3-i}} \leq MN$.   Since the $G$-action on $VX_1 \times VX_2$ is free, we have $G^+_{x_i} \cap   G_{x_{3-i}}=\{1\}$. Applying Proposition~\ref{prop:Goursat} to the group $ MN   \cong M \times N$, we infer that one of the two groups $G^+_{x_i}$ or $G_{x_{3-i}}$ coincides with one of the two simple factors of $MN$. Both of those factors are normal subgroups of $G$. It finally follows that $G^+_{x_i}$ or $G_{x_{3-i}}$ is normal in $G$. Hence $G^+_{x_i}$ acts trivially on $VX_i$ or $G_{x_{3-i}}$ acts trivially on $VX_{3-i}$. This implies that $|VX_i| \leq 2$ or $|VX_{3-i}| \leq 2$, which contradicts Lemma~\ref{lem:Basic} since $d_1, d_2 \geq 3$. 
\end{proof}

\subsection{Assuming that  $N$ acts freely   on $VX_i$ but not on $VX_{3-i}$} 

\begin{lem}\label{lem:NonFree:1}
Let $i \in \{1, 2\}$. Assume that $N$ acts freely on    $VX_i$ and non-freely on $VX_{3-i}$. Then $G_{x_i}^+ := G_{x_i} \cap NG_{x_{3-i}}$ has index $1$ or $2$ in $G_{x_i}$  and $|G_{x_i}^+| = |N : N_{x_{3-i}}|$ divides $|G_{x_{3-i}} : N_{x_{3-i}}|$. Furthermore we have $|G_{x_i}| < |G_{x_{3-i}}|$. 
\end{lem}
\begin{proof}
The $N$-orbits on $VX_i$ define a $G$-invariant partition into subsets of size $|N|$. Since $G_{x_{3-i}}$ acts regularly on $VX_i$ by Lemma~\ref{lem:Basic}, we infer that $|N|$ divides $|G_{x_{3-i}}|$.  We have $|G_{x_{3-i}}| = |G_{x_{3-i}} : N_{x_{3-i}}| |N_{x_{3-i}}|$ and $|N| = |N : N_{x_{3-i}}| |N_{x_{3-i}}|$. It follows that  $ |N : N_{x_{3-i}}|$ divides $|G_{x_{3-i}} : N_{x_{3-i}}|$.

By Lemma~\ref{lem:Basic}, the group    $G_{x_i}$ acts regularly  on $VX_{3-i}$. The hypotheses imply that $N$ has at most two orbits on $VX_{3-i}$ by Corollary~\ref{cor:BuMo}, so that $NG_{x_{3-i}}$ has index $1$ or $2$ in $G$. Accordingly, the group  $G_{x_i}^+ = G_{x_i} \cap NG_{x_{3-i}}$ has index $1$ or $2$ in $G_{x_i}$, and we have $|G_{x_i}^+| = |N : N_{x_{3-i}}|$. 

Since $N_{x_{3-i}}$ is transitive on $E(x_{3-i})$  (see Lemma~\ref{lem:BuMo:LocallyQuasiPrimitive} and Corollary~\ref{cor:BuMo}), we have $|N_{x_{3-i}} | \geq d_{3-i} \geq 3$.  Thus $|G_{x_i}| \leq  2|G_{x_i}^+|  \leq  2|G_{x_{3-i}} : N_{x_{3-i}}|  \leq  2|G_{x_{3-i}}|/3 <  |G_{x_{3-i}}|$. 
\end{proof}

\subsection{A minimality condition} 

We shall now consider   $(X_1, X_2, G) \in \mathcal F(F_1, F_2)$ satisfying the following:
\begin{description}
	\item[(Min)] For all $(X'_1, X'_2, G') \in \mathcal F(F_1, F_2)$, we have $|VX'_1 \times VX'_2 | \geq |VX_1 \times VX_2 | $. 
\end{description}


\begin{lem}\label{lem:Min}
Let $(X_1, X_2, G) \in \mathcal F(F_1, F_2)$ satisfy (Min). Then there is $i \in \{1, 2\}$ such that the $N$-action  on the graph $X_i$  is not free. 
\end{lem}
\begin{proof}
	If the $N$-action on $X_i$ were free for $i=1$ and $2$, then the triple 
	$$(N\backslash X_1, N\backslash X_2, G/N)$$ 
	would belong to $\mathcal F(F_1, F_2)$ by Lemma~\ref{lem:Kernel} (see also Proposition~\ref{prop:Infinite-Finite}), which would contradict the hypothesis (Min). 	
\end{proof}

\begin{lem}\label{lem:NonFree:2}
Let $(X_1, X_2, G) \in \mathcal F(F_1, F_2)$ satisfy (Min). 
Assume moreover that there is $i \in \{1, 2\}$ such that $N$ does not act freely on    $VX_i$. Then $\mathrm{C}_G(N)=\{1\}$. In particular $N$ is not abelian. Moreover $\max \{d_1, d_2\} \geq 5$. 
\end{lem}
\begin{proof}
Assume that $\mathrm{C}_G(N) \neq \{1\}$. Let $M \neq \{1\}$ be a minimal normal subgroup of $G$ contained in $\mathrm{C}_G(N)$. 	We aim at finding a contradiction. 

Suppose first that $M$ acts freely on both $VX_1$ and $VX_2$. By Lemma~\ref{lem:Min}, the $M$-action cannot be free on the set of geometric edges of both $X_1$ and $X_2$. 

Assume  that $M$  is not free on $X_i$. Then, by Lemma~\ref{lem:Nfree}, the $M$-action is regular on $VX_i$. Since $[M, N]=\{1\}$,  the group $N_{x_i}$ fixes pointwise the $M$-orbit of $x_i$. Thus $N_{x_i}$ acts trivially on $VX_i$. On the other hand $N$ has at most two orbits on $VX_i$ by Corollary~\ref{cor:BuMo}, since the $N$-action on $VX_i$ is not free by hypothesis. It follows that $|VX_i|  \leq  2$, so that $|G_{x_{3-i}}| \leq  2$ by Lemma~\ref{lem:Basic}, which is absurd.
 
Assume now that $M$ is not free on $X_{3-i}$. Then $|G_{x_i}|$ divides $|G_{x_{3-i}}|$ by Lemma~\ref{lem:DoublyFree}(ii). In particular $|G_{x_i}| \leq |G_{x_{3-i}}|$, so that  the $N$-action cannot be free on $VX_{3-i}$ by Lemma~\ref{lem:NonFree:1}. We may thus apply the same argument as in the preceding paragraph to conclude that $|VX_{3-i}| \leq  2$, leading to a contradiction.  

This shows that the $M$-action cannot be free on both $VX_1$ and $VX_2$. 

Assume first that the $M$-action is not free on $VX_i$. We may then invoke Corollary~\ref{cor:BuMo:2}, which ensures that $X_i$ is the complete bipartite graph $\mathbf K_{d_i, d_i}$ and that $M$ and $N$ both act  regularly on the set of geometric edges of $X_i$. In particular $N_{x_i}$ acts regularly on $E(x_i)$. It follows that the $2$-transitive permutation group $F_i$ has a regular normal subgroup. Thus $F_i$  is of affine type and $d_i$ is a prime power. We also have $|G_{x_{3-i}}| = |VX_i | = 2d_i$.  Since $G_{x_{3-i}}$ admits a $2$-transitive permutation action on $d_{3-i}$ points, we deduce that $d_{3-i}(d_{3-i}-1)$ divides $2d_i$. Since $d_i$ is a prime power, we must have $ d_{3-i}= 3$ and $d_i$ is a power of $3$. Applying Corollary~\ref{cor:TroWeiss} to the stabilizer $G_{x_{3-i}}$, which has order~$2d_i$, we deduce that $d_i=3$.   Therefore  $|G_{x_{3-i}}| = |VX_i | = 6$ and hence $G_{x_{3-i}} \cong \Sym(3)$. Moreover $|G_{x_i}| \in \{6, 12\}$ since $X_i \cong \mathbf K_{3, 3}$.

Since $M$ and $N$ both act  regularly on the set of geometric edges of $X_i$, we have $M = N \cong \mathbf C_3^2$. Notice that $N$ is a $3$-Sylow subgroup of $G$. It has $4$~cyclic subgroups of order~$3$, which are permuted by $G$.  Since $N$ is minimal normal in $G$, that permutation action must be fixed-point-free. Each involution in $G$ has $0$, $2$ or $4$ fixed points in $VX_i$, and if some involution fixes $4$ vertices, then $G_{x_i}^{[1]}$ is non-trivial. 
Assume now that $|G_{x_i}| = 6$, so that $G_{x_i} \cong \Sym(3)$ and $G_{x_i}^{[1]}=\{1\}$. Then  every involution  $\sigma \in G_{x_i}$ has $2$ fixed points on $VX_i$. It follows that its conjugation action on $N$ maps each element on its inverse, and hence it acts trivially on the set of cyclic subgroups  of $N$. This implies that the cyclic subgroup of $G_{x_{3-i}}$ of order~$3$ is normalized by both $G_{x_1}$ and $G_{x_2}$. Thus it is normal in $G$ by Lemma~\ref{lem:Basic}, contradicting the minimality of $N$. We deduce that $|G_{x_i}| = 12$. 
Hence  $X_{3-i}$ is a $3$-regular graph with $12$~vertices which is also a Cayley graph for  $G_{x_i}$  on which the group $G$ acts locally $2$-transitively.  Such a graph does not exist by \cite[Theorem~1.1]{LiLu}.

We conclude finally that  the $M$-action is free on $VX_i$, and  non-free on $VX_{3-i}$.  Lemma~\ref{lem:NonFree:1} successively implies that $|G_{x_i}| < |G_{x_{3-i}}|$, and  that the $N$-action on $VX_{3-i}$ cannot be free. We may finish the proof by swapping $X_1$ and $X_2$ and use the same argument as in the previous paragraph.  This confirms that $\mathrm{C}_G(N) = \triv$. 

If $d_1, d_2 \leq 4$, then  the only prime divisors of $|G|$ would be $2$ and $3$, so that $G$ would be solvable and $N$ abelian, a contradiction. 
\end{proof}	

\subsection{Assuming that $|F_i|$ has only two prime divisors}

\begin{lem}\label{lem:F2primes}
Let $(X_1, X_2, G) \in \mathcal F(F_1, F_2)$  satisfy (Min). 
Assume that there is $i \in \{1, 2\}$ such that $|F_i|$ has only two primes divisors.  Then the $N$-action on $VX_{3-i}$ is not free. 
\end{lem}
\begin{proof}
The hypothesis on $F_i$ implies that $|G_{x_i}|$ has only two primes divisors.  Suppose for a contradiction that  $N$ acts freely on $VX_{3-i}$. Then $|N|$ divides  $|VX_{3-i}| = |G_{x_i}|$ by Lemma~\ref{lem:Basic}. So the characteristically simple group $N$ must be abelian. Hence $N$ acts freely on $VX_i$ by Lemma~\ref{lem:NonFree:2}, and also freely on the set of geometric edges of $X_1$ and $X_2$ by Lemma~\ref{lem:DoublyFree}(iii). This contradicts Lemma~\ref{lem:Min}. 
\end{proof}

\subsection{Assuming that $F_1 \cong \mathbf{C}_5 \rtimes \mathbf{C}_4$}

\begin{lem}\label{lem:C5C4} 
	Let $(X_1, X_2, G) \in \mathcal F(F_1, F_2)$ satisfy (Min). Assume that $d_1 = 5$, $F_1 \cong \mathbf{C}_5 \rtimes \mathbf{C}_4$. Then  $d_2 \geq 5$. 
\end{lem}
\begin{proof}
Suppose for a contradiction that $d_2 \leq 	4$. In particular $F_2$ is a $\{2, 3\}$-group, and $G_{x_2}$ is a  $\{2, 3\}$-group as well.  Moreover the hypothesis on $F_1$ implies that $G_{x_1}$ is a $\{2, 5\}$-group whose order is not divisible by~$25$. In particular   $G$ is a $\{2, 3, 5\}$-group whose order is divisible by $5$ but not by $25$ in view of  Lemma~\ref{lem:Basic}. 

Lemma~\ref{lem:F2primes} ensures that the  $N$-actions on $VX_1$ and on  $VX_2$ are both non-free. Thus $\mathrm C_G(N)=\{1\}$  by Lemma~\ref{lem:NonFree:2}; in particular $N$ is not abelian. Thus $5$ divides $|N|$, and since $25$ does not divide $|G|$, we infer that $N$ is simple non-abelian and that $G$ is almost simple. 
From Proposition~\ref{prop:HL}, we have $N \cong \Alt(5), \Alt(6)$ or $\mathrm U_4(2)$. Moreover Lemma~\ref{lem:Basic} affords a  factorization $G = G_{x_1} G_{x_2}$ of $G$ as a product of two solvable subgroups. 

If $N \cong \Alt(5)$, then $|G| = 60$ or $120$, while $|G_{x_1}| = 20, 40$ or $80$ by Corollary~\ref{cor:TroWeiss}. Therefore $|G_{x_2}| \leq 6$, whence $d_2 = 3$ and  $G_{x_2} \cong \Sym(3)$. Hence the graph $X_2$, which is a Cayley graph for $G_{x_1}$, contradicts \cite[Theorem~1.1]{LiLu} in that case. 

If $N \cong \Alt(6) \cong \PSL_2(\FF_9)$, we invoke  \cite[Proposition~4.1]{LiXia} and deduce that $G_{x_2}$ has a  normal $3$-Sylow subgroup, whose order is~$9$. Thus $d_2=4$ by Corollary~\ref{cor:TroWeiss}, and we get a contradiction since a finite group with  a  normal $3$-Sylow subgroup cannot have a quotient isomorphic to $\Alt(4)$ or $\Sym(4)$. 

Finally, if $N \cong \mathrm U_4(2)$, then \cite[Proposition~4.1]{LiXia}  ensures that $G_{x_1}$ has a normal subgroup isomorphic to $\mathbf C_2^4$. Since the only normal $2$-group in $G_{x_1}/G_{x_1}^{[1]} \cong \mathbf{C}_5 \rtimes \mathbf{C}_4$ is the trivial one, we deduce that $G_{x_1}^{[1]}$ contains a subgroup isomorphic to $\mathbf  C_2^4$. By Corollary~\ref{cor:AlmostSimplePointStab}, the group $G_{x_1}^{[1]}$ is isomorphic to a subgroup of a point stabilizer in $F_1$. In particular $|G_{x_1}^{[1]}| \leq 4$, a contradiction. 
\end{proof}

\subsection{If $N$ is not simple then $X_1$ is a complete bipartite graph}

\begin{lem}\label{lem:N-non-simple}
Let $(X_1, X_2, G) \in \mathcal F(F_1, F_2)$ satisfy (Min). Assume that:
\begin{enumerate}[(1)]
	\item  There is $i \in \{1, 2\}$ such that $N$ does not act freely on    $VX_i$.
	\item For $j = 1$ and $2$, if $d_j \geq 7$ then   $F_j \geq \Alt(d_j)$.
	\item  $N$ is not simple.
	\item $d_1 \geq d_2$.
\end{enumerate}
Then $X_1$ is the complete bipartite graph $\mathbf K_{d_1, d_1}$  and  one of  the following conditions holds: 
\begin{enumerate}[(i)]
\item $d_2 = 3$ and $d_1 = 24$.

\item $d_2 = 4$, and $ d_1 \in \big\{6n \mid n \geq 2  \text{ divides } 2^2\cdot 3^5\big\}$.  	

\item $d_2 =  5$, $F_2 \cong \mathbf{C}_5 \rtimes \mathbf{C}_4$ and  $d_1 \in \big\{10, 20, 40\big\}$. 	

\item $d_2 =  5$,  $\soc(F_2) \cong \Alt(5)$ and  $d_1 \in \big\{30n \mid n \geq 2 \text{ divides } 2^8 \cdot 3\big\}$. 	

\item $d_2 = 6$, $\soc(F_2) \cong \Alt(5)$, and $d_1 \in \big\{30n  \mid n \geq 2 \text{ divides } 2^3 \cdot 5^2\big\}$.

\item $d_2 \geq 6$, and $d_1 \in \big\{\frac{d_2!} 2,  \frac{d_2!(d_2-1)!} 4, \frac{d_2!(d_2-1)!} 2\big\}$. 

\end{enumerate}
\end{lem}
\begin{proof}
The group $N$ is characteristically simple,  so that $N = S_1 \times \dots \times S_k$, where $S_i  $ is isomorphic to a  finite simple group $S$ for all $i$. Moreover $S$ is not abelian  and $d_1 \geq 5$ by Lemma~\ref{lem:NonFree:2}. The condition (3) ensures that $k \geq 2$. Furthermore we have $C_G(N)=\{1\}$ by Lemma~\ref{lem:NonFree:2}. 

The first step is to  establish the following. 

\begin{claim*}
There is $j \in \{1, 2\}$ such that $d_j \geq 5$, $F_j \not \cong \mathbf{C}_5 \rtimes \mathbf{C}_4$ and $N$ does not act freely on    $VX_j$.
\end{claim*}

By Lemma~\ref{lem:NonFree:2}, we have $d_1 \geq 5$. Assume that $j=1$ does not satisfy the claim. Then  either $N$ acts freely on $VX_1$, or $N$ does not act freely on $VX_1$ and  $F_1 \cong \mathbf{C}_5 \rtimes \mathbf{C}_4$. 

If $N$ acts freely on $VX_1$, then it acts non-freely on $VX_2$ by (1), and it follows from Lemma~\ref{lem:NonFree:1} that $|G_{x_1}|$ has a subgroup of index at most~$2$ whose order divides $|G_{x_2} : N_{x_2}|$. Since $d_1 \geq 5$, it follows that $|G_{x_1}|$, and thus also  $|G_{x_2}|$ is divisible by~$5$. In particular $d_2 \geq 5$. Moreover $|N|$ divides $|VX_1|$ which is equal to $|G_{x_2}|$ by Lemma~\ref{lem:Basic}.  Thus  $|G_{x_2}|$ has at least~$3$ prime divisors (because $N$ is not solvable). In particular $F_2 \not \cong \mathbf{C}_5 \rtimes \mathbf{C}_4$. Thus $j=2$ satisfies the claim in this case. 

Assume now that $N$ does not act freely on $VX_1$ and that $F_1 \cong \mathbf{C}_5 \rtimes \mathbf{C}_4$. Hence $d_2 \leq d_1 = 5$ by (4). In view of Lemm~\ref{lem:C5C4}, we have $d_2 = 5$. Moreover  $N$ does not act freely on $VX_2$ by Lemma~\ref{lem:F2primes}. It follows that  $\soc(F_2) = \Alt(5)$ since otherwise $G_{x_1}$ and $G_{x_2}$ would both be $\{2, 5\}$-groups, contradicting that $N$ is non-abelian. Thus $j=2$ satisfies the claim in this case as well. This ends the proof of the claim.

\medskip
In view of the claim, we may, upon replacing $i$ by $3-i$,  strengthen the hypothesis~(1) and assume in addition that $d_i \geq 5$ and that $F_i \not \cong \mathbf{C}_5 \rtimes \mathbf{C}_4$. In particular $\soc(F_i)$ is simple and $2$-transitive.

Assume next that the $S_1$-action on $VX_i$ is not free. In particular the $S_j$-action on $VX_i$ is not free for all $j \in \{1, \dots, k\}$ since the simple factors of $N$ are permuted transitively under the conjugation action of $G$. 

Since $d_i \geq 5$ and $\soc(F_i)$ is simple,  we know that the socle of $N_{v}/N_{v}^{[1]}$ is simple and $2$-transitive on $E(v)$ for every vertex $v \in VX_i$. For $j \neq m \in \{1, \dots, k\} $ and any $v \in VX_i$, it follows that if $(S_j)_v$ is non-trivial on $E(v)$ then $(S_m)_v$ is trivial on $E(v)$. 
We now apply Lemma~\ref{lem:BuMo:LocallyQuasiPrimitive} to  the normal subgroups $S_j$ and $S_m$ of $N$. For each of them we get a bipartition of $X_i$, and the previous observation together with the fact that $S_j$ and $S_m$ are conjugate in $G$ implies that $(S_j)_v$ is non-trivial on $E(v)$ if and only if $(S_m)_v$ is trivial on $E(v)$ for all $v \in VX_i$. Since this holds for all pairs $j \neq m \in \{1, \dots, k\} $, it follows that $k=2$. Given two adjacent vertices $v, w$ such that $(S_1)_v$ is non-trivial on $E(v)$, we know infer that $(S_1)_v$ fixes all neighbours of $w$ and $(S_2)_w$ fixes all neighbours of $v$. Using that $(S_1)_v$ is transitive on the neighbours of $v$ (resp.  $(S_2)_w$ is transitive on the neighbours of $w$) we deduce that $X_i$ is the complete bipartite graph $K_{d_i, d_i}$. Since $d_{3-i} \geq 3$ and $G_{x_{3-i}}$ has a $2$-transitive action on a set of cardinatliy $d_{3-i}$,  we obtain
$$2d_{3-i} \leq d_{3-i} (d_{3-i} - 1) \leq |G_{x_{3-i}}| = |VX_i| = 2d_i.$$
Hence $d_i = \max \{d_1, d_2\} = d_1$. Moreover the equality case $d_1 = d_2$ occurs only if $d_1 = d_2= 3$, which  is impossible since $d_1 \geq 5$. It follows that $d_1 > d_2$, hence $i = 1$. So $X_1$ is the complete bipartite graph $\mathbf K_{d_1, d_1}$. Moreover $G_{x_1}/G_{x_1}^{[1]} \cong F_1$ is almost simple, with socle $= \Alt(d_1)$ if $d_1 \geq 7$.  

If  $d_2 = 3$, we invoke  \cite[Theorem~1.1]{LiLu} and deduce that   $d_1 = 24$. 

If $d_2 \geq 4$ we use the fact that  the order of $G_{x_2}$, which is equal to $|VX_1| = 2d_1$, is subject to Corollary~\ref{cor:TroWeiss}. This provides numerical constraints on $(d_1, d_2)$. Those  can be slightly strengthened by observing that $G_{x_2}$ acts vertex-transitively on the complete bipartite graph $\mathbf K_{d_1,d_1}$, and thus possesses a subgroup of index~$2$. In particular $G_{x_2}$ cannot be $\Alt(d_2)$  or $\Alt(d_2) \times \Alt(d_2-1)$ for all $d_2 \geq 4$.  The required conditions (i)--(vi) follow.

\medskip 
We assume henceforth that  the $S_1$-action on $VX_i$ is   free.  In particular $|S|$ divides $|VX_i | = |G_{x_{3-i}}|$, so that $d_{3-i} \geq 5$ and $F_{3-i} \not \cong \mathbf{C}_5 \rtimes \mathbf{C}_4$.  In particular, if  the action of $S_1$ (hence of $N$) on $VX_{3-i}$ is not free, then $j=1$ and $2$ both satisfy the claim above, and we may thus argue as in the case already treated. 

\medskip 
It remains to consider the case where   all simple factors of $N$ act freely on both $VX_1$ and $VX_2$, since $G$ permutes transitively the simple factors of $N$.  In particular $|S|$ divides $|VX_j| = |G_{x_{3-j}}|$ for $j = 1$ and $2$, hence  $d_1 \geq d_2 \geq 5$ and $F_1 \not \cong \mathbf{C}_5\rtimes \mathbf{C}_4 \not \cong F_2$.  In particular $F_1$ and $F_2$ are both almost simple by the hypothesis (2). 

The rest of the proof aims at reaching a contradiction, thereby showing that the only possible situation is the one we have just described. We distinguish two cases. 

\medskip \noindent
\textbf{Case (1).} \textit{	$d_1\leq 6$}.  
\medskip

Then the only prime divisors of $|G_{x_1}|$ and $|G_{x_2}|$ are $2$, $3$ and $5$. Thus the same holds for $|G|$, whence also $|S|$, by Lemma~\ref{lem:Basic}.  Moreover $\soc(F_1)$ and $\soc(F_2)$ are isomorphic to $A_5 \cong \PSL_2(\FF_5)$ (acting on $5$ or $6$ points) or $A_6$ (acting on $6$ points), see Table~\ref{tab:Small2Trans}. By Corollary~\ref{cor:TroWeiss}, this implies that $3^4$ does not divide  $|G_{x_j}|$ for $j=1$ and $2$. In particular $3^7$ does not divide $|G|$ by Lemma~\ref{lem:Basic}, so that $S  \cong \Alt(5)$ or $\Alt(6)$ by Proposition~\ref{prop:HL}. 
 
We claim that if $N$ acts freely on $VX_1$, then $d_1 = d_2 = 6$. Indeed, if the claim fails, then $N$ acts freely on $VX_1$ and $d_2 = 5$ since $5 \leq d_2 \leq d_1 \leq 6$. The hypothesis (1) implies that $N$ does not act freely on $VX_2$. By Lemma~\ref{lem:NonFree:1}, the stabilizer $G_{x_1}$ has a subgroup $G_{x_1}^+$ of index at most~$2$ whose order divides $|G_{x_2} : N_{x_2}|$. Since $d_2 = 5$, the group $G_{x_2}/G_{x_2}^{[1]}$ is almost simple with socle $\Alt(5)$. Moreover $N_{x_2}/N_{x_2}^{[1]}$ contains the socle of $G_{x_2}/G_{x_2}^{[1]}$ since the $N$-action on $VX_2$ is not free. Using Corollary~\ref{cor:TroWeiss}, we deduce that $|G_{x_2} : N_{x_2}|$ is not divisible by $5$, whereas $5$ divides $|G_{x_1}^+|$. This is a contradiction. 

In view of the claim, we may, upon swapping the indices $1$ and $2$, assume that the $N$-action on $VX_1$ is not free. As observed above, we have $d_j \in \{5, 6\}$ and $\soc(F_j) \in \{\Alt(5), \Alt(6)\}$ for $j =1, 2$ in the case at hand. 
We shall now consider  three cases  successively  namely  $(d_1, \soc(F_1))) = (6, \Alt(6))$, $(6, \Alt(5))$ or $(5, \Alt(5))$. 

If $(d_1, \soc(F_1))) = (6, \Alt(6))$, then $N_{x_1}/N_{x_1}^{[1]}$ contains $\Alt(6)$, so that $S \cong \Alt(6)$. In particular $|N|$, hence also $|G|$, is divisible by $3^{2k}$. We have already seen that $|G|$ is not divisible by $3^7$, so that  $k \leq 3$.  Corollary~\ref{cor:AlmostSimplePointStab} ensures that  $G_{x_1}^{[1]}$ is either trivial, or almost simple with socle isomorphic to $\Alt(5)$. In particular there is no homomorphism $G_{x_1} \to \Sym(3)$ with transitive image. Recall from Corollary~\ref{cor:BuMo} that $N$ has at most two orbits on $VX_1$. In particular $|G : G_{x_1}N| \leq 2$. Since the conjugation action of $G$ permutes transitively the simple factors of $N$, the case $k=3$ is impossible, and we have $k=2$.  
Hence $N = S_1 \times S_2 \cong \Alt(6) \times \Alt(6)$ and we know that $N_{x_1}/N_{x_1}^{[1]}$ is almost simple with socle $\Alt(6)$. Considering the projection of $N_{x_1}$ on the simple factors of $N$, we deduce that image of $N_{x_1}^{[1]}$ under at least one of these projections must be trivial. In other words $N_{x_1}^{[1]}$ is contained in one of the two simple factors of $N$. We have seen above that all simple factors of $N$ act freely on $VX_1$. Therefore   $N_{x_1}^{[1]} = \triv$. Thus $G_{x_1}^{[1]}  \cap N= \triv$, so that $G_{x_1}^{[1]}$ embeds in the quotient group $G/N$. Since $C_G(N) = \triv$ and $N \cong \Alt(6) \times \Alt(6)$, the quotient $G/N$ embeds in $(\Out(\Alt(6)) \times \Out(\Alt(6)) )\rtimes \Sym(2)$, which is a $2$-group. On the other hand, by Corollary~\ref{cor:AlmostSimplePointStab}, the group $G_{x_1}^{[1]}$ is either trivial or almost simple (with socle $\Alt(5)$). We deduce that $G_{x_1}^{[1]} = \triv$.  Therefore we have $N_{x_1} \cong \Alt(6)$ and $|G_{x_1} : N_{x_1}| \leq 2$, so that the $N$-action on $VX_2$ is not free by Lemma~\ref{lem:NonFree:1}. We now distinguish 
$3$~subcases. 


If $(d_2, \soc(F_2)) = (6, \Alt(6))$, then by symmetry we have $N_{x_2} \cong \Alt(6)$, and it then follows from Proposition~\ref{prop:Goursat} that $N_{x_1} \cap N_{x_2}$ is non-trivial. This is absurd since the $G$-action on $VX_1 \times VX_2$ is free. 

If $(d_2, \soc(F_2)) = (6, \Alt(5))$, then $|G_{x_2}|$ is not divisible by $3^2$ in view of Corollary~\ref{cor:TroWeiss}, and we obtain a contradiction since $|N|$, whence also $|G|$, is divisible by $3^4$. 

If $(d_2, \soc(F_2))  = (5, \Alt(5))$, we consider the group $H= NG_{x_1}$, which is of index at most~$2$ in $G$ since the $N$-action on $VX_1$ has at most $2$~orbits. The local action of $H$ on $VX_2$ is $\Alt(5)$ or $\Sym(5)$, so $H_{x_2}/H_{x_2}^{[1]} = \Alt(5)$ or $\Sym(5)$. Moreover $|H_{x_2}| = |N : N_{x_1}| = 2^3.3^2.5$, so that $|H_{x_2}^{[1]}|= 3$ or $6$. On the other hand,  consider a vertex $y_2 \in VX_2$ adjacent to $x_2$. By Theorem~\ref{thm:TrofimovWeiss}(ii), the group $H_{x_2}^{[1]} \cap H_{y_2}^{[1]}$ is a $2$-group. Therefore the natural image of $H_{x_2}^{[1]} $ in $H_{y_2}/ H_{y_2}^{[1]}$ is non-trivial. Moreover it is isomorphic to a  normal subgroup of a point stabilizer in $\Alt(5)$ or $\Sym(5)$. Since the latter groups are $3$-transitive, it follows that the order of $H_{x_2}^{[1]} $ is divisible by $4$, a contradiction.   This finishes the case $(d_1, \soc(F_1)) = (6, \Alt(6))$.

If $(d_1, \soc(F_1))  =(6, \Alt(5))$, then $|G_{x_1}|$ is not divisible by $3^2$ in view of Corollary~\ref{cor:TroWeiss}. It follows that the $N$-action on $VX_2$ cannot be free, since otherwise $|N|$ would divide $|G_{x_1}| = |VX_2|$, so the latter would be divisible by $3^k \geq 3^2$. We may thus assume that $\soc(F_2)  \cong \Alt(5)$, since otherwise $(d_2, \soc(F_2)) = (6, \Alt(6))$ and we may swap $X_1$ and $X_2$ and invoke the case that has already been treated. If $d_2 = 6$, then $|G_{x_2}|$ is not divisible by $3^2$ by Corollary~\ref{cor:TroWeiss}, so that $|G|$ is not divisible by $3^3$. This yields $k=2$. If $d_2 = 5$, then $|G_{x_2}|$ is not divisible by $3^3$ by Corollary~\ref{cor:TroWeiss}, so that $|G|$ is not divisible by $3^4$. Thus $k \leq 3$, but if $k=3$, then $|N|$ is divisible by $3^3$ and $|G/N|$ is divisible by $3$ since $G$ permutes transitively the simple factors of $N$. Since  $|G|$ is not divisible by $3^4$, we obtain $k = 2$ in all cases. If $S \cong \Alt(6)$, then $|N|$ is divisible dy $3^4$, which is impossible. So $S \cong \Alt(5)$ and $N \cong \Alt(5) \times \Alt(5)$. Since the $N$-action on both $VX_1$ and $VX_2$ is non-free, it follows that  $N_{x_j}/N_{x_j}^{[1]}$ contains $\Alt(5)$ for $j=1, 2$. Since the simple factors of $N$ act freely on $VX_1$ and $VX_2$, we have $N_{x_j}^{[1]} = \{1\}$. Using again Proposition~\ref{prop:Goursat}, we deduce that  $N_{x_1} \cap N_{x_2}$ is non-trivial, a contradiction.  

If $(d_1, \soc(F_1)) =(5, \Alt(5))$, then $|G_{x_1}|$ is not divisible by $5^2$. It follows that the $N$-action on $VX_2$ cannot be free, since otherwise $|N|$ would divide $|G_{x_1}| = |VX_2|$, so the latter would be divisible by $5^k \geq 5^2$. Moreover we have   $(d_2, \soc(F_2)) = (5, \Alt(5))$, since $d_1 \geq d_2 \geq 5$.   In particular $|G_{x_2}|$ is not divisible by $5^2$, hence $k=2$. We cannot have $S \cong \Alt(5)$, since otherwise we would get the same contradiction as in the previous paragraph. Thus $S \cong \Alt(6)$. Thus $|G_{x_1}^{[1]}|$ and $|G_{x_2}^{[1]}|$ are both divisible by~$3$ since $|G| = |G_{x_1}||G_{x_2}|$. Thus $|N_{x_1}^{[1]}|$ is divisible by $3$ since otherwise $|G/N|$ would be divisible by~$3$. This is not the case since $C_G(N) = \triv$, so that the quotient $G/N$ embeds in $(\Out(\Alt(6)) \times \Out(\Alt(6)) )\rtimes \Sym(2)$, which is a $2$-group. We now consider the projection of $N_{x_1}$ to each simple factor $S_j$ of $N$. Since $N_{x_1}/N_{x_1}^{[1]}$ contains $\Alt(5)$ and since the only subgroups of $\Alt(6)$ containing a subnormal subgroup isomorphic to $\Alt(5)$ are $\Alt(5)$ and $\Alt(6)$, we deduce that $N_{x_1}^{[1]}$ is contained in one of the two simple factors of $N$. This is impossible,  since all the simple factors of $N$ act freely on $VX_1$. This proves that the case $d_1 \leq 6$ does not occur.

\medskip \noindent
\textbf{Case (2):} \textit{	$d_1\geq  7$}.  
\medskip

We claim that if $N$ acts freely on $VX_1$, then $d_1 = d_2$. Indeed, if the claim fails, then $N$ acts freely on $VX_1$ and $d_1 >d_2$.  The hypothesis (1) implies that $N$ does not act freely on $VX_2$. By Lemma~\ref{lem:NonFree:1}, the stabilizer $G_{x_1}$ has a subgroup $G_{x_1}^+$ of index at most~$2$ whose order divides $|G_{x_2} : N_{x_2}|$. By the discussion directly preceding Case~(1), the group $G_{x_2}/G_{x_2}^{[1]}$ is almost simple. Moreover $N_{x_2}/N_{x_2}^{[1]}$ contains the socle of $G_{x_2}/G_{x_2}^{[1]}$ since the $N$-action on $VX_2$ is not free. Since $d_1 \geq 7$,  we deduce from the hypothesis (2) that $7$ divides $|G_{x_1}^+|$, hence also $|G_{x_2} : N_{x_2}|$. It then follows from Corollary~\ref{cor:TroWeiss} that $d_2 \geq 8$. We may then invoke Corollary~\ref{cor:AlmostSimplePointStab} to establish that $\frac{d_1!} 4$ divides $|G_{x_1}^+|$, and that $|G_{x_2} : N_{x_2}|$ divides $\frac{d_2! (d_2-1)!}{d_2!/2} = 2 (d_2-1)!$. Since   $|G_{x_1}^+|$ divides $|G_{x_2} : N_{x_2}|$, we deduce that $d_1 = d_2 = 8$. The claim follows.

In view of the claim, we may, upon swapping the indices $1$ and $2$, assume that the $N$-action on $VX_1$ is not free.

For   $j \in \{1, 2\}$, if the permutation group $F_j$ has almost simple stabilizers,  then Corollary~\ref{cor:AlmostSimplePointStab} ensures that 
$$|G_{x_j}| \leq d_j! (d_j-1)! \leq d_1! (d_1-1)!.$$
This holds in particular for $j=1$. If the point stabilizers in $F_2$ are not almost simple, then the discussion directly preceding Case~(1)  implies  that either $(d_2, \soc(F_2)) = (5, \Alt(5))$ or $(d_2, \soc(F_2)) = (6, \Alt(5))$. In all cases, we invoke Corollary~\ref{cor:TroWeiss}, which respectively yields the following upper bounds: 
$$|G_{x_2}| \leq 5! 4! 4^4$$
if  $(d_2, \soc(F_2)) = (5, \Alt(5))$,  or
$$|G_{x_2}| \leq 5! 5! 5$$
if $(d_2, \soc(F_2)) = (6, \Alt(5))$. 
In either case, we obtain
$$|G_{x_2}| \leq 7! 6! \leq d_1! (d_1-1)!.$$
This proves that $|G| = |G_{x_1}| |G_{x_2}| \leq d_1!^2 (d_1-1)!^2$. On the other hand we know that $N_{x_1}/N_{x_1}^{[1]}$ contains the socle of $F_1$, which is the alternating group $\Alt(d_1)$ in the case at hand. Considering the projection of $N_{x_1}$ to each of the simple factors of $N$, we infer that $d_1!/2 = |\Alt(d_1) | \leq |S|$. This yields
$$\frac{d_1!^k}{2^k}  \leq |S|^k = |N | \leq |G| \leq d_1!^2 (d_1-1)!^2.$$ 
We deduce that $k \leq 3$. In particular $\Out(N)$ is solvable, so that the $N$-action on $VX_2$ is not free since otherwise $G_{x_2}$ would map injectively in $\Out(N)$ since $C_G(N) = \triv$, contradicting that $G_{x_2}$ has a non-abelian simple subquotient. Moreover,  the group $NG_{x_1}$ has index at most~$2$ in $G$ by Corollary~\ref{cor:BuMo}, and  $G$ permutes transitively the   $k$ simple factors of $N$. Thus, if $k = 3$ then $G_{x_1}$ has a transitive action on a $3$-point set. However, by Corollary~\ref{cor:AlmostSimplePointStab}, the group $G_{x_1}$ does not have any subgroup of index~$3$. Thus  $k=2$. Hence $N = S_1 \times S_2 \cong S \times S$. 

The $N$-action on both $VX_1$ and $VX_2$ is non-free, hence each has at most $2$ orbits. Recall moreover that the  $S_1$- and $S_2$-actions on both $VX_1$ and $VX_2$ are  all free. In particular $|S|$ divides both $|VX_1| = |G_{x_2}|$ and $|VX_2| = |G_{x_1}|$. 

Assume that $G_{x_1}^{[1]}$ is non-trivial. Then it is almost simple with socle $\Alt(d_1-1)$ by Corollary~\ref{cor:AlmostSimplePointStab}. Therefore so is $N \cap G_{x_1}^{[1]} = N_{x_1}^{[1]}$, since $C_G(N)= \triv$   and $\Out(N)$ is solvable.  Since both simple factors of $N$ act freely on $VX_1$, we see that the projection map $N \to S_1$ yields an injective homomorphism of $N_{x_1}$ into $S$. Since $|S|$ divides $|G_{x_1}|$, we obtain that 
$\frac{d_1!} 2 \frac{(d_1-1)!} 2$ divides $|S|$, which in turn divides $d_1! (d_1-1)!$. It follows that the image of $N_{x_1}$ into $S$ has index at most~$4$. Since $S$ is simple, the image of $N_{x_1}$ into $S$ must be surjective, which is absurd since the normal subgroup $N_{x_1}^{[1]}$ is non-trivial. This proves that $G_{x_1}^{[1]} =\triv$. 

Invoking again that $N_{x_1}$ maps injectively to $S$ and that  $|S|$   divides $|G_{x_1}| \in \{d_1!, d_1!/2\}$, we now deduce  that $S \cong \Alt(d_1) \cong N_{x_1}$. Since $N$ has at most $2$ orbits on $VX_1$, we deduce that  $|G_{x_2}| = |VX_1| \in \{|N : N_{x_1}|, 2|N : N_{x_1}|\} = \{d_1!/2, d_1!\}$. In particular $d_2 \geq 7$. We may thus apply the same arguments for $G_{x_2}$ as for $G_{x_1}$ in the previous paragraph to establish that $G_{x_2}^{[1]} =\triv$ and  that $ N_{x_2} \cong  \Alt(d_1) \cong N_{x_1}$. Since $N = S_1 \times S_2 \cong  \Alt(d_1) \times \Alt(d_1)$, we deduce from Proposition~\ref{prop:Goursat} that $N_{x_1} \cap N_{x_2}$ is non-trivial.  Therefore the $G$-action on $VX_1 \times VX_2$ is not free. This final contradiction finishes the proof. 
\end{proof}

\subsection{If $N$ is simple then $X_1$ is a complete graph} 

\begin{lem}\label{lem:N-simple}
Let $(X_1, X_2, G) \in \mathcal F(F_1, F_2)$ satisfy (Min). Assume that:
\begin{enumerate}[(1)]
	\item  There is $i \in \{1, 2\}$ such that $N$ does not act freely on    $VX_i$.
	\item For $j = 1$ and $2$, if $d_j \geq 7$ then   $F_j \geq \Alt(d_j)$.
	\item  $N$ is  simple.
	\item $d_1 \geq d_2$.
\end{enumerate}
Then $X_1$ is the complete graph $\mathbf K_{d_1+1}$, and one of the following conditions holds:
\begin{enumerate}[(i)]
\item $d_2 = 3$, and $d_1 \in \big\{23, 47\big\} $. 

\item $d_2 = 4$, and $ d_1 \in \big\{ 12n -1\mid n \geq 2  \text{ divides }  2^2\cdot 3^5\big\}$.  	

\item $d_2 =  5$,  $F_2 \cong \mathbf{C}_5 \rtimes \mathbf{C}_4$ and  $d_1 \in \big\{ 19, 39, 79\big\}$. 	

\item $d_2 =  5$,  $\soc(F_2) \cong \Alt(5)$ and  $d_1 \in \big\{  60n-1\mid n \text{ divides }  2^8 \cdot 3\big\}$. 	

\item $d_2 = 6$, $\soc(F_2) \cong \Alt(5)$, and $d_1 \in \big\{60n-1 \mid n \text{ divides } 2^3 \cdot 5^2\big\}$.

\item $d_2 \geq 6$, and $d_1 \in \big\{  \frac{d_2!} 2 -1, d_2!-1, \frac{d_2!(d_2-1)!} 4-1, \frac{d_2!(d_2-1)!} 2-1,d_2!(d_2-1)!-1\big\}$. 
\end{enumerate}
\end{lem}

\begin{proof}
We know that $N$ is non-abelian and  that $d_1 \geq 5$   by Lemma~\ref{lem:NonFree:2}. That lemma ensures that $\mathrm C_G(N) = \{1\}$, so that $G$ is almost simple with socle $N$. 

If $d_1 = 5$ and $F_1 \cong \mathbf{C}_5 \rtimes \mathbf{C}_4$, then $d_2 = 5$ by Lemma~\ref{lem:C5C4}. In that case  $F_2 \not \cong \mathbf{C}_5 \rtimes \mathbf{C}_4$ since otherwise $G$ would be a $\{2, 5\}$-group by Lemma~\ref{lem:Basic}, hence solvable, a contradiction. Therefore, upon replacing  $(X_1, X_2, G)$ by $(X_2, X_1, G)$ in the case $d_1= d_2= 5$, we may assume without loss of generality that $F_1$ is almost simple.  
In particular $G_{x_1}$ is not solvable, hence $N \cap G_{x_1} \neq \triv$ since $C_G(N) = \triv$ and $\Out(N)$ is solvable. Thus $N$ does not act freely on $VX_1$.  

Since $d_1 \geq d_2$,  the hypothesis (2) implies that $\pi(G_{x_2}) \subseteq \pi(G_{x_1})$ with the notation of Section~\ref{sec:LPS},  so that $\pi(G) = \pi(G_{x_1})$. Moreover $N$ is not contained in $G_{x_1}$, since $G$ acts faithfully on $X_1$. Thus all the hypotheses of Corollary~\ref{cor:LPS} are satisfied. 

We shall now consider successively the seven exceptional cases of Corollary~\ref{cor:LPS} displayed in Table~\ref{tab:LPS} and show that each of them does not occur. An observation that we shall used repeatedly is the following. Table~\ref{tab:LPS} provides us with the possible values of the index $|N: N_{x_1}|$. We know moreover that $N$ has at most two orbits on $VX_1$ (by Corollary~\ref{cor:BuMo}) and the $G_{x_2}$ acts regularly on $VX_1$ (by Lemma~\ref{lem:Basic}). Thus $|VX_1| = |G_{x_2} |$ equals $|N: N_{x_1}|$ or $2|N: N_{x_1}|$. This can be confronted with Corollary~\ref{cor:TroWeiss}, which provides independent constraints that the number $|G_{x_2} |$ must satisfy. 

The numbering of the cases below is chosen according to the numbering of the rows in Table~\ref{tab:LPS}. 

\medskip \noindent
\textbf{Case (1).} \textit{$N= \Alt(6)$ and $N_{x_1} = \PSL_2(\FF_5)$}. 
\medskip

By Corollary~\ref{cor:LPS}, we have $|N: N_{x_1}| = 6$. Hence  $|G_{x_2}| = |VX_1| \in\{ 6, 12\}$, so that $d_2 \leq 4$. From \cite[Theorem~1.1]{LiLu}, we deduce that $d_2 \neq 3$. Thus $|VX_1| = 12$,  and $d_2 = 4$. Hence $G_{x_2} \cong \Alt(4)$, so $G_{x_2}$ does not have any subgroup of index~$2$. But $N$ acts with two orbits on $VX_1$, so that $NG_{x_1}$ is an index~$2$ subgroup of $G$, and $G_{x_2} \cap NG_{x_1}$ is an index~$2$ subgroup of $G_{x_2}$ by Lemma~\ref{lem:Basic}. This is a contradiction. 

\medskip \noindent
\textbf{Case (2).} \textit{$N= \mathrm U_3(5)$ and $N_{x_1} =\Alt(7)$}. 
\medskip

By Corollary~\ref{cor:LPS}, we have $|N : N_{x_1}| = 2 \cdot 5^2$. Hence  $|G_{x_2}| = |VX_1| \in\{ 50, 100\}$. This is impossible by  Corollary~\ref{cor:TroWeiss}.

\medskip \noindent
\textbf{Case (3).} \textit{$N= \mathrm U_4(2)$  and $N_{x_1} \leq 2^4.\Alt(5)$ or $N_{x_1} \leq \Sym(6)$}. 
\medskip

Then the only primes dividing $|N : N_{x_1}| $ are $2$ and $3$, so that $G_{x_2}$ is a $\{2, 3\}$-group. In particular it is solvable, and $d_2\in \{3, 4\}$. We may thus invoke \cite[Theorem~1.1]{LiXia}; it follows that the triple $(G, G_{x_2}, G_{x_1})$ must be as in row~10 or~11 of \cite[Table~1.2]{LiXia}. In the former case we have $N_{x_1} = 2^4.\Alt(5)$, so that $|G_{x_2}| = 27$ or $54$. This is impossible by Corollary~\ref{cor:TroWeiss}. Thus  the triple $(G, G_{x_2}, G_{x_1})$ is as in row~11 of \cite[Table~1.2]{LiXia}, and $N_{x_1} = \Alt(5), \Sym(5), \Alt(6)$ or $\Sym(6)$.  If $N$ is transitive on $VX_1$, then the hypotheses of \cite[Lemma~8.30]{LiXia} are satisfied and we get a contradiction. Thus $|G:NG_{x_1}|=2$. Since $\Out(N)$ is of order~$2$, we deduce that $N = NG_{x_1}$. In particular $G_{x_1} \leq N$ and $N_{x_2}$ is of index~$2$ in $G_{x_2}$. 
Moreover, the information provided  by   \cite[Table~1.2]{LiXia} ensures that $N_{x_2}$ is a subgroup of $3_+^{1+2}: 2.\Alt(4)$, which is a parabolic subgroup of $\mathrm{PSp}_4(3) \cong \mathrm U_4(2)$.  Observe that the natural action of $\Alt(4)$ on the Heisenberg group $3_+^{1+2}$ does not preserve any subgroup of order~$3^2$; therefore  the largest power of~$3$ dividing $|N_{x_2}|$ (and hence also $G_{x_2}$) cannot be $3^3$. On the other hand, we have $|N_{x_2}| = |N:N_{x_1}|$. We deduce that $N_{x_1} = G_{x_1}$ can neither be $\Alt(5)$ nor $\Sym(5)$, so that it is $\Alt(6)$ or $\Sym(6)$. The latter possibility is excluded because $\Sym(6)$ is maximal  in $N$, and the factorization $N = N_{x_1} N_{x_2}$ would then contradict \cite[Theorem~1.1]{LPS2}. Thus $N_{x_1} = G_{x_1} = \Alt(6)$. It follows that $|N_{x_2}| = 2^3\cdot 3^2$, hence $N_{x_2} \cong 3:2.\Alt(4)$. It follows that $N$ is locally $2$-transitive on $X_2$ with local action at every vertex isomorphic to $\Alt(4)$ by Lemma~\ref{lem:BuMo:LocallyQuasiPrimitive}. It follows that the  point stabilizers in $N_{v}/N_{v}^{[1]}$ are cyclic of order~$3$ for all $v\in VX_2$, so that $N_{x_2}^{[1]}$ is a $3$-group.  This contradicts that $N_{x_2} \cong 3:2.\Alt(4)$.

\medskip \noindent
\textbf{Case (4).} \textit{$N= \mathrm{U}_4(3)$ and $N_{x_1} = \Alt(7)$}.  
\medskip

Then $|N : N_{x_1}| = 2^4 \cdot 3^4$. Thus $G_{x_2}$ is a $\{2, 3\}$-group, hence solvable.    We may thus invoke \cite[Theorem~1.1]{LiXia}, which yields a contradiction. 

\medskip \noindent
\textbf{Case (5).} \textit{$N= \mathrm{PSp}_4(7)$ and $N_{x_1} = \Alt(7)$}.  
\medskip

Then $|N : N_{x_1}| = 2^5 \cdot 5 \cdot 7^3$, so that $|G_{x_2}|  \in \{|N : N_{x_1}|, 2|N : N_{x_1}|\} $ violates Corollary~\ref{cor:TroWeiss}.

\medskip \noindent
\textbf{Case (6).} \textit{$N= \mathrm{Sp}_6(2)$ and $N_{x_1} = \Alt(7), \Sym(7), \Alt(8)$ or $\Sym(8)$}. 
\medskip

 Here again   $G_{x_2}$ is a $\{2, 3\}$-group, hence solvable.  Since $\Out(N)$ is trivial in this case, we have $G=N$ so that the hypotheses of   \cite[Lemma~8.30]{LiXia} are satisfied. The latter result yields a contradiction. 

\medskip \noindent
\textbf{Case (7).} \textit{$N= \mathrm{P}\Omega^+_8(2)$ and $N_{x_1} \leq P_1, P_3, P_4$ or $N_{x_1} \leq \Alt(9)$}. 
\medskip

Then $G_{x_2}$ is a $\{2, 3, 5\}$-group whose order is divisible by~$30$, so that $d_2  \in \{5, 6\}$ and $F_2 \not \cong \mathbf{C}_5 \rtimes \mathbf{C}_4$. 

Let us first consider the case where $N_{x_1}$ is contained in a parabolic subgroup $P_k$. Then the socle of $N_{x_1}/N_{x_1}^{[1]}$ must be isomorphic to the Levi factor of $P_k$, which is $\mathrm{SL}_4(\FF_2) \cong \Alt(8)$. It follows that $|N:N_{x_1}|$ is divisible by $3^3\cdot 5$, but $|N:N_{x_1}|$ is not divisible by $25$. This contradicts Corollary~\ref{cor:TroWeiss} for the $|G_{x_2}|$. 

We now assume that $N_{x_1} \leq \Alt(9)$. If $N_{x_1}$ is a proper subgroup of $\Alt(9)$, the same numerical considerations as in the case   $ N_{x_1} \leq P_k$ yield a contradiction. It follows that 
$N_{x_1} = \Alt(9)$, so $|G_{x_2}| = 2^a\cdot 3 \cdot 5$ with $a=6$ or $7$. Using Corollary~\ref{cor:TroWeiss}, we infer that $d_2= 5$, so $F_2 = \Alt(5)$ or $\Sym(5)$ because $F_2 \not \cong \mathbf C_5 \rtimes \mathbf C_4$.  Notice that $N_{x_1}$ is a maximal subgroup of $N$ in the case at hand. Therefore $G_{x_1}$ is a maximal subgroup of the almost simple group $NG_{x_1}$. Denoting $G_{x_2}^+ = G_{x_2} \cap NG_{x_1}$, the factorization $G = G_{x_1} G_{x_2}$ yields a factorization $NG_{x_1} = G_{x_1} G_{x_2}^+$ since $N$ has at most two orbits on $VX_1$. We may then invoke \cite[Theorem~1.1]{LPS2}, which ensures that $G_{x_2}^+ = 2^4.\Alt(5)$. In particular $|G_{x_2}| /60 = 2^4$ or $2^5$.  Recall that $\Alt(5) \cong \Omega_4^-(2)$. A more precise look at the structure of $G_{x_2}^+$ afforded by that factorization reveals that  $G_{x_2}^+ \cong \FF_2^4\rtimes \Omega_4^-(2)$, where $\FF_2^4$ is the standard $\Omega_4^-(2)$-module:  that information can be extracted from Examples (h) and (i) and Lemma~10.7 in \cite{Baumeister}.

On the other hand, Corollary~\ref{cor:TroWeiss} yields $d_2 = 5$, and we may invoke \cite[Theorem~(1.2)]{Weiss79} to elucidate the structure of $G_{x_2}$. Given the possible values for the order of $G_{x_2}$, we must have $s=4$ in the notation of \cite[Theorem~(1.2)]{Weiss79}, so that the latter result yields an embedding of $G_{x_2}$ as a subgroup of 
$\mathbf F_4^2 \rtimes \mathrm P\Gamma \mathrm L_2(\mathbf F_4)$ containing  $\mathbf F_4^2 \rtimes \mathrm{PSL}_2(\mathbf F_4)$, where the action of $\Alt(5) \cong \mathrm{PSL}_2(\mathbf F_4)$ on $\mathbf F_4^2$ is the standard one. That embedding must map $G_{x_2}^+ $ isomorphically onto $\mathbf F_4^2 \rtimes \mathrm{PSL}_2(\mathbf F_4)$. This is a contradiction, because the groups   $\FF_2^4\rtimes \Omega_4^-(2)$ and $\mathbf F_4^2 \rtimes \mathrm{PSL}_2(\mathbf F_4)$ are not isomorphic (the two corresponding modules of $\Alt(5) \cong \Omega_4^-(2) \cong \mathrm{PSL}_2(\mathbf F_4)$ are not isomorphic). 

\medskip
Since all the seven exceptional cases of Corollary~\ref{cor:LPS} are excluded, we deduce from the latter result that $N \cong \Alt(c)$ and $\Alt(k) \lhd N_{x_1} \leq \Sym(k) \times \Sym(c-k)$, where $k \leq c$ are integers such that $p \leq k$ for every prime $p \leq c$. Moreover $c\geq 5$ because  $d_1 \geq 5$, and the case $c=5$ is excluded since it would imply that $N \leq G_{x_1}$. 

If $c=6$, then $N_{x_1} = \Alt(5)$ and $|G_{x_2}| = 6$ or $12$,  and we obtain a contradiction with the same arguments as in Case~(1) above. We assume henceforth that 
$$c \geq 7.$$ 
Hence $G = \Alt(c)$ or $\Sym(c)$. Using the existence of a prime $p$ with $ \frac {c+1} 2< p \leq c$ (see \cite[1.1]{WiWi} for a more general fact), we  deduce from Corollary~\ref{cor:AlmostSimplePointStab} that $G_{x_1}^{[1]} =\triv$, so that $G_{x_1} = \Alt(d_1)$ or $\Sym(d_1)$. 

We shall now use the fact that  the factorization $G = G_{x_1} G_{x_2}$ must be described by the main results from \cite{WiWi}. 

If $G = \Alt(c)$, we invoke \cite[Theorem~A]{WiWi}. Case III from \cite[Theorem~A]{WiWi} is impossible since $G_{x_1} = \Alt(d_1)$ or $\Sym(d_1)$.  Case II is also impossible in view of our hypotheses on $F_2$ (special care is required in view of the isomorphism $\Alt(6) \cong \PSL_2(\FF_9)$; however $\PSL_2(\FF_q)$ appears in \cite[Theorem~A, Case~II]{WiWi} only for prime powers $q$ congruent to $3$ modulo~$4$). Thus we are in Case I of  \cite[Theorem~A]{WiWi}. This yields  $G_{x_1} = \Alt(d_1) = \Alt(k)$ and $G_{x_2}$ acts sharply $t$-transitively on $\{1, \dots, c\}$, where $t = c-k$. 

If $G = \Sym(c)$, we invoke \cite[Theorem~S]{WiWi}. Using similar arguments, we obtain $d_1=k$ and either $G_{x_2}$ or its index~$2$ subgroup $N_{x_2}$ acts sharply $t$-transitively on $\{1, \dots, c\}$, where $t = c-k$. 

\medskip
We next claim that $t=1$. In order to establish this, we assume that $t \geq 2$ and discuss the value of $d_2$. We shall repeatedly use the fact that  a sharply $t$-transitive group on a set of cardinality $c$ is of order $c\cdot (c-1) \dots (c-t+1)$. 

If $d_2 = 3$, then Corollary~\ref{cor:TroWeiss} yields $c =3$ or $4$, which is absurd. 

If $d_2 = 4$, then Corollary~\ref{cor:TroWeiss} yields $c =9$ and $t=2$. It then follows that $G_{x_2}$ or $N_{x_2}$ is the affine group $\FF_9 \rtimes \FF_9^*$, which is absurd since $G_{x_2}/G_{x_2}^{[1]}$ is $\Alt(4)$ or $\Sym(4)$. 

If $d_2 = 5$, then $F_2 \not \cong \mathbf{C}_5 \rtimes \mathbf{C}_4$ since $c \geq 7$. Thus $\soc(F_2) \cong \Alt(5)$ and $G_{x_2}$ is a $\{2, 3,  5\}$-group.
Assume now that $G_{x_2}^{[1]} \neq \{1\}$. It then follows from  Theorem~\ref{thm:TrofimovWeiss} that $G_{x_2}^{[1]} $ has a non-trivial normal $2$-subgroup. Therefore the same holds for $G_{x_2}$. Since $G_{x_2}$ is $2$-transitive on $\{1, \dots, c\}$, it follows that $c$ is a power of~$2$. In view of \cite[Table~7.3]{Cameron}, we must have $c= 16$ since $G_{x_2}/G_{x_2}^{[1]} $ is isomorphic to $\Alt(5)$ or $\Sym(5)$. We deduce that $t=2$ (since otherwise $|G_{x_2}|$ would be divisible by~$7$), and we get a contradiction since the only sharply $2$-transitive groups on $16$ points are solvable. Thus $G_{x_2}^{[1]} = \triv$ and $G_{x_2} \cong \Alt(5)$ or $\Sym(5)$. Neither of these two groups has a  $t$-transitive action on a set of $c\geq 7$ points. 

If $d_2 = 6$ and $\soc(F_2)= \PSL_2(\FF_5) \cong \Alt(5)$, then $G_{x_2}$ is a again a $\{2, 3,  5\}$-group.  If $G_{x_2}^{[1]}\neq \triv$ then Theorem~\ref{thm:TrofimovWeiss}  ensures that $O_5(G_{x_2}^{[1]})$ is of order~$5$ or $25$. Hence $O_5(G_{x_2})$ is also of order $5$ or $25$. Since $G_{x_2}$ is $t$-transitive on $\{1, \dots, c\}$ and $c \geq 7$, we obtain  $c = 25=|O_5(G_{x_2})|$. Moreover $t=2$ since otherwise $|G_{x_2}|$ would be divisible by $23$.  Therefore $|G_{x_2}/O_5(G_{x_2})|= 24$, which is absurd since  $O_5(G_{x_2}) \leq G_{x_2}^{[1]}$. This contradiction shows that $G_{x_2}^{[1]}= \triv$, so that  $G_{x_2} \cong \Alt(5)$ or $\Sym(5)$.  As before, we arrive at a contradiction since neither of these two groups has a  $t$-transitive action on a set of $c\geq 7$ points. 

If $d_2 \geq 6$ and $\soc(F_2)=\Alt(d_2)$, then $G_{x_2}^{[1]}$ is either trivial or almost simple with socle $\Alt(d_2-1)$ by Corollary~\ref{cor:AlmostSimplePointStab}. In the latter case $G_{x_2}$ has two commuting normal subgroups of order $\frac{d_2!} 2$ and $\frac{(d_2-1)!} 2$ respectively. This prevents $G_{x_2}$  from admitting any faithful $2$-transitive action (since both normal subgroups would have to act freely and transitively, contradicting the fact that they have different orders). Hence $G_{x_2}^{[1]} = \triv$, so $G_{x_2}= \Alt(d_2)$ or $\Sym(d_2)$. In view of \cite[Table~7.4]{Cameron}, the only $2$-transitive action of the latter is the natural action on $d_2$ points, unless $d_2 = 6$, in which case there is a $2$-transitive action on $10$ points via the exceptional isomorphism $\Alt(6) \cong \PSL_2(\FF_9)$. In that case we must have $G_{x_2} = \Sym(6)$, $c=10$ and $t=3$, so $d_1 = c-t = 7$ and  $d_2= 6$. 

In order to exclude that case, we observe that by Lemma~\ref{lem:Basic}, the $7$-regular graph $X_1$ is a Cayley graph of $G_{x_2}$. The corresponding generating set of $G_{x_2}$ must thus contain an involution $\tau$ (because $7$ is odd) that maps $x_1$  to a neighbouring vertex $y_1$. Thus $\tau$ normalizes $G_{x_1, y_1}$. Notice that $G_{x_1, y_1} \cong \Alt(6)$ or $\Sym(6)$ since $G_{x_1} \cong \Alt(7)$ or $\Sym(7)$. Moreover $\langle G_{x_1} \cup \{\tau\} \rangle$  is transitive on $VX_1$, and is thus the whole group $G$. Consider that $G_{x_1}$-action on $\{1, \dots, 10\}$ given through the isomorphism $G \cong \Alt(10)$ or $\Sym(10)$. Upon reordering we may assume that the largest orbit of $G_{x_1}$ is $\{1, \dots, 7\}$ and that $G_{x_1, y_1}$ fixes the point $1$. Since we also know that $G_{x_1} \cong \Alt(7)$ or $\Sym(7)$, we deduce from \cite[Theorems~A and S]{WiWi} that $G_{x_1}$ acts trivially on $\{8,9,10\}$. Since $\tau$ normalizes $G_{x_1, y_1}$ which is isomorphic to $\Alt(6)$ or $\Sym(6)$, it must stabilize the set $\{2, \dots, 7\}$. Thus, the set $\{8, 9, 10\} \setminus \{\tau(1)\}$, which is of size $2$ or $3$, is invariant under both $G_{x_1}$ and $\tau$. This contradicts the fact that $G =\langle G_{x_1} \cup \{\tau\} \rangle$.  

\medskip
This finally shows that $t=1$. Thus $G_{x_1} = \Alt(d_1)$ or $\Sym(d_1)$ and $G = \Alt(d_1+1)$ or $\Sym(d_1+1)$, and the group $G_{x_2}$ acts regularly on a set of cardinality $d_1+1$, so $ |VX_1|= |G_{x_2}| = d_1 + 1$. It follows that $X_1$ is the complete graph $\mathbf K_{d_1 +1}$.   The numerical constraints satisfied by the pair $(d_1, d_2)$ follow from the fact that the order of $G_{x_2}$  is subject to Corollary~\ref{cor:TroWeiss}.  Furthermore, in case $d_2 = 3$, the more precise conclusion that $d_1 \in \{23, 47\}$ follows from \cite{LiLu}. 

It finally remains to exclude the case $(d_1, d_2) = (11,4)$.   In that case $G_{x_2} \cong \Alt(4)$ and $G = \Sym(12)$ or $\Alt(12)$, and $G_{x_1} = \Sym(11)$ or $\Alt(11)$. For such a triple $(G, G_{x_1}, G_{x_2})$, we deduce from Lemma~\ref{lem:Basic} that there   exists an element $g \in G_{x_1}$ such that:
\begin{itemize}
\item  $g^{-1} \in G_{x_2} g G_{x_2}$,
\item  $|G_{x_2} \backslash G_{x_2} g G_{x_2}| = 4$, and
\item  $G =  \langle \{g\} \cup  G_{x_2} \rangle$.
\end{itemize}
Using GAP, we enumerated all elements of $\Sym(11)$ and checked that none of them satisfies all of these three conditions. 
\end{proof}	

\subsection{Proofs of Theorems~\ref{thm:MainFinite} and~\ref{thm:Main}}\label{sec:proofs}

We are now ready to finish the proofs of   the main  results of this paper.

\begin{proof}[Proof of Theorem~\ref{thm:MainFinite}]
Retain the notation introduced in Section~\ref{sec:not}. Under the hypotheses of Theorem~\ref{thm:MainFinite}, we have $d_1 \geq d_2$ and the permutation group $F_i \leq \Sym(d_i)$ is $2$-transitive. Moreover $F_i$ contains $\Alt(d_i)$ if $d_i \geq 7$.  We need to show that if the set $\mathcal F(F_1,F_2)$ is non-empty, then $(d_1, d_2)$ satisfy the constraints listed in the statement of the theorem. 

Assume that $\mathcal F(F_1,F_2)$ is non-empty. We may then choose $(X_1, X_2, G) \in \mathcal F(F_1,F_2)$ satisfying (Min). Let   $N$ be a minimal normal subgroup of $G$. Then $N$ does not act freely on both $X_1$ and $X_2$ by Lemma~\ref{lem:Min}. Moreover there is $i \in \{1, 2\}$ such that $N$ does not act freely on $VX_i$ by Lemma~\ref{lem:DoublyFree:2cases}.   If $N$ is simple, then Lemma~\ref{lem:N-simple} applies, while if $N$ is not simple, we invoke Lemma~\ref{lem:N-non-simple}. In either case the required conclusion follows.  
\end{proof}

\begin{proof}[Proof of Theorem~\ref{thm:Main}]
Let $\Gamma \leq \Aut(T_1)\times \Aut(T_2)$ and assume that $\Gamma$ is reducible. We must show that $(d_1, d_2)$ satisfies the required constraints. 
	
For $i=1, 2$, let    $K_i$ be the projection on $\Aut(T_i)$ of the kernel of the $\Gamma$-action on $T_{3-i}$. Then $K_i$ does not contain any edge inversion by Lemma~\ref{lem:no-inversion}. We may therefore invoke Proposition~\ref{prop:Infinite-Finite}.  Since $\Gamma$ is reducible, it follows that the quotient group $\Gamma/K_1 \times K_2$ is finite, and so is the quotient graph $X_i = K_i \backslash T_i$. The   conclusion is now straightforward from Theorem~\ref{thm:MainFinite}. 
\end{proof}

\section{The just-infinite property}

In this final section, we assemble the ingredients needed to establish Corollary~\ref{cor:HJI}. 

We recall that a locally compact group is called \textbf{topologically simple} if its only closed normal subgroups are the trivial ones. 

The following fundamental result of Bader--Shalom generalizes a result of Burger--Mozes \cite[Theorem~4.1]{BuMo2}  concerning certain lattices in products of trees. Although we shall invoke the result in the context of lattices in products of trees, we do need the more general version of Bader--Shalom, whose hypotheses on the structure of the ambient group are more flexible. 

\begin{thm}[{Bader--Shalom \cite{BaSha}}] \label{thm:BaSha}
Let $G_1, G_2$ be compactly generated locally compact groups and $\Gamma \leq G_1 \times G_2$ be a cocompact lattice whose projection to $G_1$ and $G_2$ has dense image. Assume that for $i=1$ and $2$, the intersection $M_i$ of all non-identity closed normal subgroups of $G_i$ is topologically simple and that the quotient $G_i/M_i$ is compact. Then $\Gamma$ is hereditarily just-infinite. 
\end{thm}

In the context of groups acting on trees, locally compact groups satisfying the conditions appearing in Theorem~\ref{thm:BaSha} pop up naturally. This is illustrated by the following result of Burger--Mozes and Nebbia. 

\begin{thm}[{Burger--Mozes \cite{BuMo1}, Nebbia \cite{Nebbia}}] \label{thm:BuMo}
Let $T$ be a locally finite tree, all of whose vertices have degree~$\geq 3$. Let also $G \leq \Aut(T)$ be a closed subgroup. If the $G$-action on the set of ends  $\partial T$ of $T$ is $2$-transitive, then $G$ is compactly generated,  the intersection $M$ of all non-identity closed normal subgroups of $G$ is topologically simple,  and   the quotient $G/M$ is compact. Moreover the $M$-action on $\partial T$ is $2$-transitive. 
\end{thm}
\begin{proof}
This follows by combining several results from \cite{BuMo1} (see also \cite{Nebbia} for the case of regular trees). Details can be found in \cite[Proposition~2.1 and Theorem~2.2]{CDM}.
\end{proof}

A fundamental idea  of  Burger--Mozes is that, given a tree $T$ and a vertex-transitive group $G \leq \Aut(T)$, if $G$ is non-discrete and the local action of $G$ on $T$ is a suitable $2$-transitive group, then the closure $\overline G$ is $2$-transitive on $\partial T$ (see \cite[\S3.3]{BuMo1}), so that $\overline G$ is subject to Theorem~\ref{thm:BuMo}. The $2$-transitive groups considered by Burger--Mozes are those with almost simple (or quasi-simple) point stabilizers. In particular, their original arguments do not apply to $2$-transitive groups of degree~$\leq 5$. However, a similar local-to-global phenomenon can also be extracted from the work of V.~Trofimov. The following result, due to him, applies to numerous $2$-transitive local actions whose point stabilizers need not be almost simple. 

\begin{thm}[{V.~Trofimov \cite[Proposition~3.1]{Tro}}] \label{thm:Tro}
Let $X$ be a locally finite $d$-regular graph and $G \leq \Aut(X)$ be a vertex-transitive group whose local action on $X$ is the $2$-transitive group $F \leq \Sym(d)$. Assume that  every subnormal subgroup $S$ of the point stabilizer $F_1$, for which the index $|F_1 : \mathrm N_{F_1}(S)|$ divides a power of $d-1$, acts transitively on $\{2, \dots, d\}$. If $G$ is non-discrete, then $X$ is a tree and the closure $\overline G$ is $2$-transitive on the set of ends $\partial X$. 

Moreover, the above condition is satisfied if $d = q+1$ and $F$ contains a normal subgroup isomorphic to $\PSL_2(\FF_q)$, or if $F \geq \Alt(d)$. 
\end{thm}
\begin{proof}
The statement of \cite[Proposition~3.1]{Tro} ensures that $X$ is a tree. Although he does not write it explicitly, Trofimov's proof actually also shows that $\overline G$ is $2$-transitive on   $\partial X$. The fact that the condition holds in the case $d=q+1$ and $\PSL_2(\FF_q) \lhd F$ is explained in \cite[Example~3.2]{Tro}. If $F \geq \Alt(d)$ with $d \geq 6$, the condition is clearly satisfied since $\Alt(d-1)$ is simple. For $d \leq 5$, it follows from the preceding case (see Table~\ref{tab:Small2Trans}). 
\end{proof}

Combining the three theorems above, we obtain the following result. 

\begin{cor}\label{cor:ji-crit}
Let $d_1, d_2 \geq 3$, let $T_1, T_2$ be regular trees of degree $d_1, d_2$ and let $\Gamma \leq \Aut(T_1) \times \Aut(T_2)$ be a discrete subgroup acting transitively on $VT_1 \times VT_2$. Assume that for $i=1, 2$, the local action $F_i$ of $\Gamma$ on $T_i$ satisfies the condition in Theorem~\ref{thm:Tro}. If  $\Gamma$  is irreducible, then it is hereditarily just-infinite. 
\end{cor}
\begin{proof}
Since $\Gamma$ is irreducible, its   projection  $p_i \colon \Gamma \to \Aut(T_i)$ has a non-discrete image for $i=1$ and $2$ by \cite[Proposition~1.2]{BuMo2}. Let $G_i = \overline{p_i(\Gamma)}$. By Theorem~\ref{thm:Tro}, the   $G_i$-action on $\partial T_i$ is $2$-transitive. Thus Theorem~\ref{thm:BuMo} ensures that $G_1$ and $G_2$ satisfy the hypotheses of Theorem~\ref{thm:BaSha}. The conclusion follows. 
\end{proof}	

Corollary~\ref{cor:HJI} is an immediate consequence of Corollary~\ref{cor:ji-crit} (see Table~\ref{tab:Small2Trans}), recalling that a subgroup $\Gamma \leq \Aut(T_1) \times \Aut(T_2)$ is discrete if and only if the stabilizer $\Gamma_{(v_1, v_2)}$ of a vertex $(v_1, v_2) \in VT_1 \times VT_2$ is finite.


\providecommand{\bysame}{\leavevmode\hbox to3em{\hrulefill}\thinspace}
\providecommand{\MR}{\relax\ifhmode\unskip\space\fi MR }
\providecommand{\MRhref}[2]{%
  \href{http://www.ams.org/mathscinet-getitem?mr=#1}{#2}
}
\providecommand{\href}[2]{#2}

\end{document}